\documentclass[11pt, reqno,amsmath,amsthm,amssymb,amscd]{amsart}
\usepackage{amsmath}
\usepackage{amssymb}
\usepackage{verbatim}
\usepackage{dsfont}
\usepackage{enumerate}
\usepackage[all]{xy}
\usepackage[mathscr]{eucal}

\usepackage{amsthm}
\usepackage{stmaryrd}
\usepackage{amscd}
\usepackage{amsfonts}
\usepackage{latexsym}

\usepackage{amscd}

\usepackage{graphicx}

\usepackage{amsmath}
\usepackage{mathrsfs,amssymb, amscd,amsmath,amsthm}
\usepackage[enableskew,vcentermath]{youngtab}
\usepackage{multicol}\multicolsep=0pt
\usepackage{tikz}

 \providecommand{\og}{``}
\providecommand{\fg}{''} \providecommand{\smfandname}{and}

\usepackage{amssymb}
\usepackage{mathrsfs,amssymb}
\usepackage{multicol}\multicolsep=0pt
\usepackage{pstricks,pst-node}

\usepackage[enableskew,vcentermath]{youngtab}
\usepackage[sort]{cite}
\usepackage{xcolor,graphicx}

\def\crulefill{\leavevmode\leaders\hrule height 1pt\hfill\kern 0pt}
\long\def\QUERY#1{%
\leavevmode\newline%
\noindent$\star\star\star$\thinspace\textsf{Comment/Query}\crulefill\newline%
   \space #1\newline\hbox to 120mm{\crulefill}$\star\star\star$\newline}
\newtheorem{Theorem}{Theorem}[section]
\newtheorem{Lemma}[Theorem]{Lemma}
\newtheorem{Cor}[Theorem]{Corollary}
\newtheorem{Prop}[Theorem]{Proposition}

 \theoremstyle{definition}

\newtheorem{Defn}[Theorem]{Definition}

\newtheorem{Cond}[Theorem]{Condition}
\newtheorem{rem}[Theorem]{Remark}

\newtheorem{Assum}[Theorem]{Assumption}

\numberwithin{equation}{section}
\theoremstyle{definition}

\newtheorem{THEOREM}{Theorem}

\newtheorem{CONJECTURE}{Conjecture}

\makeatletter
\def\enumerate{\begingroup\ifnum\@enumdepth>3\@toodeep\else
      \advance\@enumdepth\@ne
      \edef\@enumctr{enum\romannumeral\the\@enumdepth}%
      \topsep\z@\parskip\z@
      \list{\csname label\@enumctr\endcsname}
        {\@nmbrlisttrue\let\@listctr\@enumctr
         \parsep\z@\itemsep\z@\topsep\z@
         \setcounter{\@enumctr}{0}
         \def\makelabel##1{\hss\llap{\rm ##1}}
       }\fi}

\makeatother

\let\bar=\overline
\let\epsilon=\varepsilon
\def\({\big(}
\def\){\big)}

\def\N{\mathbb N}

\def\Z{\mathbb Z}

\def\0{\underline{0}}

\DeclareMathOperator{\End}{End}

\DeclareMathOperator{\Rad}{Rad}

\def\s{\mathfrak s}

\def\t{\mathfrak t}

\def\Hom{\text{Hom}}

\def\U{\mathbf U}

{\catcode`\|=\active
  \gdef\set#1{\mathinner{\lbrace\,{\mathcode`\|"8000%
                                   \let|\midvert #1}\,\rbrace}}
  \gdef\seT#1{\mathinner{\Big\lbrace\,{\mathcode`\|"8000%
                                   \let|\midverT #1}\,\Big\rbrace}}
}
\def\midvert{\egroup\mid\bgroup}
\def\midverT{\egroup\,\Big|\,\bgroup}

\def\Set[#1]#2|#3|{\Big\{\ #2\ \Big| \
           \vcenter{\hsize #1mm\centering #3}\Big\}}

\def\Z{\mathbb Z}
\def\C{\mathbb C}
\def\mf{\mathbf f}

\def\qed{\hfill\mbox{$\Box$}}

\def\N{\mathbb{N}}

\def\Z{\mathbb{Z}}

\def\Hom{{\rm Hom}}

\def\Set{{\rm Set}}

\def\s{\mathfrak s}%
\def\t{\mathfrak t}%
\def\Hom{\text{Hom}}%
\def\U{\mathbf U}%
\def\textsf#1{{\textit{#1}}}%

\definecolor{white}{HTML}{FFFFFF}
\definecolor{darkblue}{HTML}{111199}
\definecolor{darkgreen}{HTML}{336633}
\definecolor{darkred}{HTML}{993333}

\definecolor{darkpurple}{HTML}{995599}

\newcommand{\frp}{\mathfrak p}

\begin{document}
\baselineskip16pt
\title[{\tiny  Decomposition numbers of cyclotomic Brauer algebras }]{Decomposition numbers of the cyclotomic Brauer algebra over the complex field, II}
 
\author{Hebing  Rui}
\email{hbrui@tongji.edu.cn}
\author{Linliang Song}
\email{llsong@tongji.edu.cn}
\address{School of Mathematical Science, Tongji University, Shanghai, 200092, China.}\thanks{ H. Rui is supported partially by NSFC (grant No.  11971351).  L. Song is supported partially by NSFC (grant No.  12071346).}
\date{\today}
\sloppy

\begin{abstract}  Following Nazarov’s suggestion in~\cite{Na1}, the cyclotomic Nazarov-Wenzl algebra~\cite{AMR}  is referred to as the cyclotomic Brauer algebra.
This paper focuses on computing the decomposition numbers of the cyclotomic Brauer algebra over $\mathbb{C}$ with arbitrary parameters. We show that   these decomposition numbers can be expressed in terms of  the parabolic Kazhdan–Lusztig polynomials of type $D_n$, with a parabolic subgroup of type $A$, under Condition~\ref{keyconj}.
\end{abstract}
\maketitle
 \baselineskip16pt 
 \section{Introduction}
Throughout this paper, we work over the complex field $\C$. All algebras are assumed to be defined over $\C$.

This paper is a continuation of \cite{RS-cyc, GRX}. 
Our objective is to compute the decomposition numbers of the cyclotomic Brauer algebra  $B_{k,r}( (-1)^k \mathbf{u}_k)$, defined  as in Definition \ref{definition of Na}, for arbitrary parameters $\mathbf u_k=(u_1, u_2, \ldots, u_k)$  under  Condition~\ref{keyconj}. Here, the sign $(-1)^k$ is  included purely for technical reasons, as demonstrated in Theorem \ref{Thm:subisok2k1}, and 
 \begin{equation}\label{minusu}(-1)^k\mathbf{u}_k = ((-1)^ku_1, (-1)^ku_2, \ldots, (-1)^k u_k).\end{equation} 
We achieve this by utilizing   parabolic Kazhdan-Lusztig polynomials of type  $D_n$  with a parabolic subgroup of type  $A$~\cite{So}.
To outline  the main ideas of the paper, we first introduce some necessary concepts. 

Given two elements $a,b$ in $\mathbb C$, following \cite{ES}, we say that $a$ is $r$-disjoint from $b$ if either  $a\pm b\notin \mathbb Z$ or  $  m=a\pm b\in \mathbb Z$ and $ |m|\ge r$. We further  say that $u_{j+1}$ is $r$-disjoint from $\mathbf u_j:=(u_1, \ldots, u_j)$ if $u_{j+1}$ is disjoint from $u_i$ for all $1\le i\le j$.  

Unless otherwise stated, we will simply say that  $a$ is disjoint from $b$ rather than $a$ is $r$-disjoint from $b$  if the meaning can be understood from the context.  

The following result represents the  first step towards our goal.  It  is inspired by the influential work of Ehrig and Stroppel in  \cite{ES},  where the case  $k=1$ was studied.

\begin{THEOREM}
\label{Thm:subisok2k1}
Suppose that $u_{j+1}$ is disjoint from $\mathbf u_j$   for all $k\le j\le 2k-1$. 
   Then there exists 
 an idempotent $e$ in 
 $ B_{2k,r}( \mathbf u_{2k})$ such that $B_{k,r}((-1)^k\mathbf{u}_k) \cong eB_{2k,r}( \mathbf u_{2k})e $.
\end{THEOREM}
The key difference in the proof of Theorem \ref{Thm:subisok2k1} between the case $k=1$ in \cite{ES} and the cases $k>1$ in this paper lies in the proofs of  Proposition \ref{genefatingfu}(1) and Lemma \ref{lem:rela5}.  The approach used  for $k=1$ does not seem to  extend naturally to higher levels. Handling the  higher-level cases requires a significantly  different approach, relying on  the generating function $W_{k,1}(u)$ in \eqref{w1} and the $u$-admissible condition in \eqref{uadm}. 

 The second step toward our goal is to  demonstrate  that the algebra  $B_{2k,r}(\mathbf{u}_{2k})$  is  isomorphic to the endomorphism algebra of an appropriate  tilting module in the parabolic category  $\mathcal{O}$ of type  $D_n$,  with a parabolic subalgebra of type  $A$. When $k=1$, this result was established  by Ehrig and Stroppel in \cite{ES, ES2},  where they used projective modules instead of tilting modules. 

 To explain our approach, let 
  $\mathfrak{p} \subset \mathfrak{so}_{2n} $ denote the parabolic subalgebra  corresponding to the subset \begin{equation}\label{lever}I= \Pi\setminus\{\alpha_{p_1}, \alpha_{p_2}, \ldots, \alpha_{p_k}\},\end{equation}  where
 $0=p_0<p_1<p_2<\ldots< p_{k-1}<p_k=n$, and 
  $\Pi$ is  the set of simple roots $\alpha_i$, for all $1\le i\le n$, as defined in  \eqref{aln} for $\mathfrak{so}_{2n}$. 
  We emphasize that $I$ corresponds to  $I_1$ for the orthogonal Lie algebra $\mathfrak{so}_{2n}$ in \cite{RS-cyc, GRX}.
  Define  \begin{equation}
  \label{pdom}
  \Lambda^\frp=\{\lambda\in \mathfrak h^*\mid \langle \lambda, \alpha^\vee\rangle \in \mathbb N\ \  \text{for all $\alpha\in I$}\},\end{equation}
where $\mathfrak h^\ast $ is dual space of the standard Cartan subalgebra $\mathfrak h$ of $\mathfrak {so}_{2n}$, and $\alpha^\vee$ is the coroot of $\alpha$. 

Throughout, $V$ denotes  the natural $\mathfrak{so}_{2n}$-module. For any $\mathbf c=(c_1, c_2, \ldots, c_k)\in \mathbb C^k$, and a non-negative integer $r$,  let 
\begin{equation}\label{mir} M_{r,\mathbf c}:=M^\frp(\lambda_{\mathbf c})\otimes V^{\otimes r},\end{equation} where  $M^\frp(\lambda_{\mathbf c})$ is  the parabolic Verma module with the  highest weight 
\begin{equation}\label{deltac}\lambda_{\mathbf c}=\sum_{j=1}^{k} c_j(\epsilon_{p_{j-1}+1}+\epsilon_{p_{j-1}+2}+\ldots+\epsilon_{p_j})\in\Lambda^{\frp}.\end{equation}   
Here $\{\epsilon_1, \epsilon_2, \ldots, \epsilon_n\}$ is a basis of $\mathfrak h^\ast$, as defined in Section~3.

The second step toward our goal involves the identification of a specific simple parabolic Verma
module associated with  any given parameters $\mathbf u_k$ of the cyclotomic Brauer algebra $B_{k,r}((-1)^k\mathbf u_k)$ as follows.
\begin{THEOREM}
\label{Thm:simpleVerma2}
    For any given $\mathbf u_k\in \mathbb C^k$, define     $\mathbf c=(c_1, c_2, \ldots, c_k)\in  \mathbb C^k$, where  \begin{equation}\label{cjuj}   c_j=  u_j+p_{j-1}-n+\frac{1}{2} ,\end{equation}   for $ 1\le j\le k$ and any  $p_0, p_1, \ldots, p_k\in \mathbb C$. Then,  there exist  positive integers $0=p_0<p_1<p_2<\ldots< p_k=n$ such that 
$p_t-p_{t-1}\gg 0$, for $1\le t\le k$,  and 
  $M^\mathfrak p(\lambda_{\mathbf c})$ is simple and tilting.
\end{THEOREM}
\begin{Assum}\label{1123} The symbols   $\mathbf c$ and $p_i$, $0\le i\le  k$,  are   defined as    in Theorem~\ref{Thm:simpleVerma2}, for any  given  $\mathbf u_k\in\mathbb C^k$,  
\end{Assum}
From this point onward,  we consistently maintain Assumption~\ref{1123}. Consequently, the parabolic Verma module  $M^\mathfrak p(\lambda_{\mathbf c})$ is simple and tilting.

Next, we define   the parameters  $u_{k+1}, u_{k+2}, \ldots, u_{2k}$ as follows: 
\begin{equation} \label{uij21}   u_j  =  -c_{2k-j+1}+p_{2k-j+1}- n+\frac{1}{2},   \text{ for $k+1\leq j\leq 2k$.} \end{equation} 
We will demonstrate  that $u_{j+1}$ is disjoint from  $\mathbf u_j$ for $j=k,k+1,\ldots,2k-1$ , under Assumption \ref{1123}. 

By combining Theorem \ref{Thm:subisok2k1},
Theorem~\ref{Thm:simpleVerma2},  and \cite[Theorem 5.4]{RS-cyc}, we establish the following theorem. It   reveals   a fundamental connection between the cyclotomic Brauer algebra with arbitrary parameters  and the parabolic category $\mathcal O^\frp $
in type $D_n$, where the parabolic  subalgebra $\frp$ is  determined by  $I$ in \eqref{lever}. 

\begin{THEOREM}
\label{The-RSmain} Let  $\mathbf u_{2k}=(u_1, u_2, \ldots, u_{2k})$, where $u_{1}, u_{2}, \ldots, u_{2k}$ are specified   in \eqref{cjuj}--\eqref{uij21}. 
Then: 
\begin{enumerate}
    \item [(1)]$B_{2k,r}^{op}(\mathbf u_{2k})\cong \End_{\mathcal O^\frp }(M_{r,\mathbf c})$ as $\mathbb C$-algebras.  
    \item[(2)]   $B_{k,r}^{op}((-1)^k\mathbf u_k)\cong e\End_{\mathcal O^\frp }(M_{r,\mathbf c})e$ as $\mathbb C$-algebras, for some  idempotent $e$ of $\End_{\mathcal O^\frp }(M_{r,\mathbf c})$, where $(-1)^k \mathbf u_k$ is defined as in \eqref{minusu}. 
\end{enumerate}
 
\end{THEOREM}
 Thanks to Theorem~\ref{The-RSmain},  the decomposition numbers of the cyclotomic Brauer algebra $B_{2k,r}(\mathbf{u}_{2k})$ can, in principle,  be computed using parabolic Kazhdan–Lusztig polynomials of type $D_n$ with a parabolic subgroup of type $A$. This result   arises naturally  from  the general framework developed in \cite{AST} and \cite[Corollary~5.10]{RS-cyc}. 
 
 To explain these results explicitly, let  $(M_{r, \mathbf c}: M^\mathfrak p(\mu))$ denote   the multiplicity of the  parabolic Verma module $ M^\mathfrak p(\mu)$ 
in any parabolic Verma flag of $M_{r, \mathbf c}$.
Define \begin{equation}\label{frbar}\begin{aligned}   \mathbf F_r & =\{ \mu \in \Lambda^\mathfrak p\mid (M_{r, \mathbf c}: M^\mathfrak p(\mu))\neq 0 \},\\
\bar {\mathbf F}_r & = \{ \lambda\in \mathbf F_r\mid T^\frp(\lambda) \text{ is a direct summand of $ M_{r,\mathbf c}$}\},\\
\end{aligned}
\end{equation} 
where $ T^\frp(\lambda)$ is the indecomposable tilting module  in $\mathcal O^\frp$ 
with the   highest weight $\lambda$. 
Then, \begin{equation}\label{dec-til} M_{r, \mathbf c}=\bigoplus_{\lambda\in \bar {\mathbf F}_{r}} T^\frp(\lambda)^{\oplus n_\lambda}.\end{equation}

From \cite{AST},   there exists a cellular basis of $\End_{\mathcal O^\frp}(M_{r,\mathbf c})$ with respect to the poset $( {\mathbf F}_{r}, \leq)$, where $\leq$ denotes the dominance order  on $\mathfrak h^*$ defined by $\lambda\le \mu$ if $\mu-\lambda\in \mathbb N\Pi$. 
The corresponding left cell modules are $\{S(\lambda)\mid \lambda\in {\mathbf F}_{r}\}$, where 
$$S(\lambda)=\Hom_{\mathcal O^\frp}(M^\frp(\lambda), M_{r, \mathbf c}).$$
By \cite{GL},   there is an invariant form $\phi_\lambda$ on $S(\lambda)$ such that 
$$D(\lambda)=S(\lambda)/\Rad\phi_\lambda,$$
where $\Rad \phi_\lambda$ denotes  the radical of $\phi_\lambda$. Furthermore,  \cite{AST} establishes  that $D(\lambda)\neq 0$ if and only if $\lambda\in \bar {\mathbf F}_{r}$. The indecomposable projective modules are 
$\{P(\lambda)\mid \lambda\in \bar {\mathbf F}_{r}\}$, where 
$$P(\lambda)=\Hom_{\mathcal O^\frp}(T^\frp(\lambda), M_{r, \mathbf c}).$$
Each $P(\lambda)$ serves as the projective cover of the simple module  $D(\lambda)$. From Theorem~\ref{The-RSmain}, $S(\lambda), P(\lambda)$ and $D(\lambda)$ can be viewed as right $B_{2k, r}(\mathbf u_{2k})$-modules, where $B_{2k, r}(\mathbf u_{2k})$ is the cyclotomic Brauer algebra  described in Theorem~\ref{The-RSmain}. Thanks to  Theorem~\ref{The-RSmain} and \cite[Corollary 5.10]{RS-cyc},
we have   
      \begin{equation}
      \label{equ:decom1}
      [S(\lambda)e: D(\mu)e]=[S(\lambda): D(\mu)] =(T^{\mathfrak p}( \mu):M^{\mathfrak p}( \lambda)),    
      \end{equation}
for all $(\lambda, \mu)\in  {\mathbf F}_{r,k}\times \bar {\mathbf F}_{r,k}$,
where $[S(\lambda): D(\mu)]$, referred  to as   the decomposition number,  represents  the multiplicity of the simple module $D(\mu)$ in a composition series of $S(\lambda)$. The sets 
 ${\mathbf F}_{r,k}$ and $ \bar {\mathbf F}_{r,k}$ are defined as: 

\begin{equation}\label{barfk} \begin{aligned} {\mathbf F}_{r,k} & =\{\lambda \in {\mathbf F}_{r}\mid S(\lambda)e\neq 0   \}, \\ \bar {\mathbf F}_{r,k}& =\{\lambda\in {\mathbf F}_{r,k}\mid D(\lambda)e\neq 0\}. \end{aligned} \end{equation} 
Explicit characterizations  of $ {\mathbf F}_{r,k}$ and  $\bar {\mathbf F}_{r,k}$ are provided  in Lemma~\ref{frk} and Theorem~D, respectively. 

By  $\eqref{equ:decom1}$,  
 the decomposition numbers of $B_{k,r}((-1)^k\mathbf u_k)$ are understood only theoretically. This limitation arises  because  
 while ${\mathbf F}_{r}$ has an explicit characterization, critical information remains unavailable for ${\mathbf F}_{r,k}$,   $\bar {\mathbf F}_{r}$, $\bar {\mathbf F}_{r,k}$,  and the multiplicities $n_\lambda$ in \eqref{dec-til}.   
  Furthermore,  no direct relationship has been established between the right cell modules $S(\lambda)$ and those defined using the explicit weakly cellular basis of $B_{2k, r}(\mathbf u_{2k}) $ in \cite[Theorem~2.4]{GRX}, as the latter framework depends inherently on combinatorial definitions. 
  
   Following \cite[(1.11)]{GRX},  
we consider the  partial ordering $\prec$ on $\Lambda^\frp$, where      $\mu\prec \lambda$ indicates  the existence of  a sequence $$\mu=\gamma^0, \gamma^1, \ldots, \gamma^j=\lambda$$ in $\Lambda^\frp$  such  that, for each $1\le l \le j$,   the simple $\mathfrak {so}_{2n}$-module $L(\gamma^{l-1})$,  with the highest weight $\gamma^{l-1}$,   appears as a composition factor of $M^{\frp} (\gamma^l)$. 

The following assumption  is referred to as the  \textsf{saturated condition}, as  stated in 
\cite[(1.12)]{GRX}.
\begin{Cond}
  \label{keyconj} For any $0\le j\le r$,  $ {\mathbf F_{j}}$  is saturated with respect  to  $\prec$  in the sense that $\mu\in {\mathbf F_{j}}$ if  $\mu\prec \lambda$ for some $\lambda\in {\mathbf F_{j}}$.   
\end{Cond}

Condition~\ref{keyconj} is well-justified since, by 
 \cite[Theorem~E]{GRX}, it  holds if  
 $\lambda_{\mathbf c}$ satisfies \begin{equation}\label{keyass1} \langle \lambda_{\mathbf c}+\rho, \beta^\vee\rangle\not \in \mathbb Z_{>0}, \ \ \text{ for all $\beta\in \Phi^+\setminus \Phi_I$.}\end{equation} 
Here $\Phi^+$ is the set of positive roots associated with $\mathfrak{so}_{2n}$, and $\Phi_I=\Phi\cap \mathbb ZI$, and $\rho$ is the half sum of all positive roots in $\Phi^+$.

Let $\Lambda_{2k, r}$ be the set of pairs $(f, \lambda)$, where $\lambda=(\lambda^{(1)}, \lambda^{(2)}, \ldots, \lambda^{(2k)})$
is a $2k$-partitions of $r-2f$, and $0\le f\le \lfloor r/2\rfloor$. For each $(f, \lambda)\in \Lambda_{2k, r}$, let  $C(f, \lambda)$ denote  the right  cell module of $B_{2k, r}(\mathbf u_{2k})$,  defined using the weakly cellular basis in \cite[Theorem~2.4]{GRX}. According to \cite{GL},   there exists  an invariant form $\phi_{f, \lambda}$ on $C(f, \lambda)$. Define  $$D(f, \lambda)= C(f, \lambda)/\Rad \phi_{f, \lambda}, $$
where $\Rad \phi_{f, \lambda}$ denotes the radical of $\phi_{f, \lambda}$.
Then all non-zero $D(f, \lambda)$ form a complete set of  pairwise non-isomorphic simple  $B_{2k,r} (\mathbf u_{2k})$-modules. 

From \cite[Theorem~2.5]{GRX}, we know that  $D(f, \lambda)\neq 0$ if and only if $\lambda\in \bar\Lambda_{2k, r}$, where   $$\bar\Lambda_{2k, r}=\{(f, \lambda)\in \Lambda_{2k, r}\mid \sigma^{-1}(\lambda)  \text{ is $\mathbf {u}_{2k}$-restricted }\}. $$
Here $\sigma$ denotes the generalized Mullineaux involution as defined in \cite[Remark~5.10]{RS1}. This result follows from  the classification of simple modules for $B_{2k,r}(\mathbf u_{2k})$ \cite{RSi}, since  $\omega_0=2n\neq0$, where 
$\{\omega_i\mid i\in\N\}$ is a set of  parameters associated with  $B_{2k, r}(\mathbf u_{2k})$, and  determined by $\mathbf u_k$ (see \eqref{uadm}). Notably, $\mathbf u_{2k}$-restricted partitions are also known as   Kleshchev multipatitions (see e.g.,  \cite[(2.11)--(2.12)]{GRX} for an explicit description).

By \cite[(4.4)]{GRX}, there is a bijection $\hat \  :\    \Lambda_{2k,r}\rightarrow {\mathbf F}_{r}$. From this point onward, we use the symbol  $\lambda$ to denote either a weight in $\Lambda^\frp$ or a multipartition, with the meaning to be understood from the context.
From  \cite[Theorem~D]{GRX},   
\begin{equation}\label{isocell1}S(\hat \lambda)\cong C(f,  \lambda'),\  \text{ and }\  
\bar {\mathbf F}_{r}= \{\hat \lambda\in {\mathbf F}_{r}\mid (f,  \lambda')\in \bar \Lambda_{2k,r}\},\end{equation}
if Condition  \ref{keyconj} holds.

The following is the main result about the decomposition numbers of $B_{k,r}((-1)^k\mathbf u_k) $ under the additional assumption  $\omega_i\neq 0$ for some $0\le i\le k-1$ when $r$ is even. 
 \begin{THEOREM}\label{first123}
 Fix any $\mathbf u_k\in \C^k$.
 Suppose Condition~\ref{keyconj} holds.  Then:
 \begin{enumerate}
 \item[(1)] $S(\hat \lambda)e\neq 0$ if and only if $\hat \lambda\in {\mathbf F}_{r,k}$, where $${\mathbf F}_{r,k}=\{\mu \in {\mathbf F}_{r}\mid \delta^\mu_i\ge 0 \text{ for all } 1\le i \le n  \},$$ and $\delta^\mu$ is given in \eqref{hatmu}. Moreover, $S(\hat \lambda)e\cong C(f,\lambda')e$ and $\{S(\hat \lambda)e\mid \hat \lambda \in {\mathbf F}_{r,k}\}$  consist of all cell modules of $B_{k, r}((-1)^k \mathbf u_k)$.
 \item[(2)]  
 $D(\hat \lambda)e\neq 0$ if and only if $(f, \lambda') \in \bar\Lambda_{2k,r}$, under the assumption that $\omega_i\neq0$ for some $0\le i\le k-1$ when $r$ is even.   Consequently, 
 $ \bar {\mathbf F}_{r,k}={\mathbf F}_{r,k}\cap \bar {\mathbf F}_{r}$.
\item[(3)] $\{D(\hat \lambda)e\mid \hat \lambda\in \bar {\mathbf F}_{r,k}\}$ forms   a complete set of pair-wise non-isomorphic simple $B_{k, r}((-1)^k \mathbf u_k)$-modules, under the assumption that  $\omega_i\neq0$ for some $0\le i\le k-1$ when $r$ is even.  
\item [(4)] 
 $ [S(\hat \lambda)e: D(\hat \mu)e]= ( T^\frp(\hat \mu):M^\frp(\hat \lambda))$, for all $\hat \lambda \in  {\mathbf F}_{r,k} $, and $\hat\mu\in  \bar {\mathbf F}_{r,k}$,  under the assumption  that $\omega_i\neq0$ for some $0\le i\le k-1$ when $r$ is even.  
 \end{enumerate}    \end{THEOREM}
From Theorem~\ref{first123}, we deduce  an explicit description of $\bar {\mathbf F}_{r,k}$. Moreover, we have a concrete construction of $S(\hat \lambda)e$ and $D(
\hat \mu)e$ via $C(f,  \lambda')e$ and $D(f,  \mu')e$.   As a consequence, the decomposition numbers of $B_{k,r}((-1)^k\mathbf u_{k})$ for arbitrary parameters  can be obtained more explicitly using parabolic Kazhdan-Lusztig polynomials of type $D_n$ with a parabolic subgroup of type $A$. 

Finally,  to provide  a concrete description of Condition~\ref{keyconj}, we define  
\begin{equation}
    \label{equ:defofphiA}
    \begin{aligned}     \Phi_A& =\{\epsilon_i-\epsilon_j\mid i<j \} \cap \Phi^+\setminus \Phi_I,\ \  
    \Psi_{\lambda_\mathbf c}^+ & =\{\beta\in \Phi^+\setminus \Phi_I\mid \langle \lambda_\mathbf c+\rho,  \beta^\vee \rangle \in \mathbb Z_{>0} \}.    \end{aligned}
\end{equation}

\begin{THEOREM}\label{mainthm:E}
If 
  $\Phi_A\cap \Psi_{\lambda_{\mathbf c}}^+=\emptyset$, then Condition~\ref{keyconj} holds. In particular, when $k=1$, Condition~\ref{keyconj} holds  unconditionally.  
  \end{THEOREM}

Notably, 
\cite[Theorem~E]{GRX} is a special case of Theorem~E. 
When $k=1$, $\omega_0=0$ is equivalent to $\omega_i= 0$ for all $i\in \mathbb N$. This can be verified using \eqref{w1}-\eqref{uadm} and $u_1=\frac{1}{2} (1-\omega_0)$. 

Moreover, $\Phi_A\cap \Psi_{\lambda_{\mathbf c}}^+=\emptyset$ automatically holds in this setting. Thus, our result strengthens the work of Ehrig and Stroppel in \cite{ES, ES2}, as we provide an explicit construction of  right  cell modules for the cyclotomic Brauer algebra $B_{2, r}(\mathbf u_2)$.

\begin{CONJECTURE}
\label{conjecture} 
Suppose $\Phi_A\cap \Psi_{\lambda_{\mathbf c}}^+\neq \emptyset$. Then Condition~\ref{keyconj} is satisfied, and consequently,   Theorem \ref{first123}~(1)--(4) hold, if $r$ is odd or if  $r$  is even and $\omega_i\neq 0$ for some $0\le i\le k-1$.

\end{CONJECTURE}
It is worth mentioning
 that if Conjecture \ref{conjecture} holds, then,  combined with  Theorem \ref{mainthm:E},  this would enable us  to explicitly  compute the decomposition numbers of $B_{k,r}((-1)^k\mathbf u_k)$  for arbitrary parameter $\mathbf u_k$, provided that either  $r$ is odd or  $\omega_i\neq 0$ for some $0\le i\le k-1$ when $r$ is even.
 
The paper is organized as follows.
In Section~2, 
we recall the definition of the cyclotomic Brauer algebra over $\mathbb C$ and prove that the cyclotomic Brauer algebra $B_{k, r}((-1)^k \mathbf u_k)$ is isomorphic to an idempotent truncation  of $B_{2k, r}({\mathbf u}_{2k})$, where
$u_{k+1}, \ldots, u_{2k}$ are chosen to satisfy  a certain disjoint condition.

In Section 3, 
we provide an explicit description of the condition under which a parabolic Verma module is simple in a parabolic BBG category $\mathcal O$ associated with the orthogonal Lie algebra $\mathfrak{so}_{2n}$.  This allows us to realize  the cyclotomic Brauer algebra $B_{2k, r}(\mathbf u_{2k})$ as the endomorphism algebra of the tensor product of a suitable simple scalar-type parabolic Verma module with the natural module in the parabolic BBG category $\mathcal O$ in type $D_n$. Here 
 $u_1, u_2, \ldots, u_k$ are arbitrary parameters in $\mathbb C$, and $u_{k+1}, \ldots, u_{2k}$ are defined as in   \eqref{uij21}. 
 Moreover,  we prove that Condition~\ref{keyconj}  holds when $\Phi_A\cap \Psi_{\lambda_\mathbf c}^+=\emptyset$.
 
 In Section~4, we explicitly compute the decomposition numbers of $B_{k, r}((-1)^k \mathbf u_k)$.
 This computation holds when    $\Phi_A\cap \Psi_{\lambda_\mathbf c}^+=\emptyset$, with the additional constraints that either 
 $r$ is odd or $r$ is even,  and $\omega_i\neq 0$ for some $i$.

\section{The cyclotomic Brauer algebra and its idempotent subalgebra }
This section focuses  on realizing  the cyclotomic Brauer algebra  $B_{k,r}((-1)^k\mathbf u_k)$ as an idempotent  truncation  of 
the higher level cyclotomic Brauer algebra  $B_{2k,r}(\mathbf u_{2k})$.  Here the   parameters $u_{k+1},\ldots, u_{2k}$ are selected to   satisfy a specific compatibility condition.
 This  foundational result constitutes the critical first step in our broader objective of computing the decomposition numbers of $B_{k,r}((-1)^k\mathbf u_k)$.
\subsection{Cyclotomic Brauer algebra}
\begin{Defn}
\label{definition of Na}
\cite[Definition~2.13]{AMR} Given two positive  integers $k$ and $r$, and two families of parameters   
 $\mathbf u_k=(u_1, \ldots, u_k)\in \mathbb C^k$, and $\omega=(\omega_i)\in \C^\mathbb N$, the cyclotomic  Nazarov-Wenzl algebra  $ B_{k, r}(\mathbf u_k)$ is the unital associative $\mathbb C$-algebra with  generators  
 $\{e_i, s_i, x_j\mid 1 \le i\le r-1,
1\le j\leq r\}$ and    relations
\begin{multicols}{2}
		\begin{enumerate}
			\item [(1)] $s_i^2=1$,  
			\item[(2)] $s_is_j=s_js_i$, for $|i-j|>1$,
			\item[(3)] $s_is_{i+1}s_i\!=\!s_{i+1}s_is_{i+1}$, 
			\item[(4)] $s_ix_j=x_js_i$, for $j\neq i,i+1$,
			\item[(5)]  $e_1x_1^ke_1=\omega_k e_1, \forall k\in \mathbb N$,
			\item[(6)] $s_ie_j=e_js_i$, for  $|i-j|>1$,
            \item[(7)] $e_ie_j=e_je_i$, for  $|i-j|>1$,
			\item[(8)]  $e_ix_j =x_j e_i$, for  $j\neq i,i+1$,
			\item [(9)] $x_ix_j=x_jx_i$,
			
			\item[(10)]  $s_ix_i-x_{i+1}s_i=e_i-1$,
			\item[(11)]  $x_i s_i-s_i x_{i+1}=e_i-1 $, 
			\item[(12)] $e_i s_i=e_i=s_ie_i$,
			\item[(13)] $s_ie_{i+1}e_i=s_{i+1}e_i$,
			\item[(14)] $e_i e_{i+1}s_i =e_i s_{i+1}$, 
			\item[(15)] $e_i e_{i+1}e_i =e_{i+1}$,
			\item[(16)] $ e_{i+1}e_i e_{i+1} =e_i$, 
			\item [(17)]  $e_i(x_i+x_{i+1})=(x_i+x_{i+1})e_i=0$,
 			\item[(18)]  $(x_1-u_1)(x_1-u_2)\cdots (x_1-u_k)=0$.
		\end{enumerate}
	\end{multicols}
 \end{Defn}

The cyclotomic Nazarov-Wenzl algebra is defined as a quotient of the affine Wenzl algebra in \cite[\S4]{Na}. Following  Nazarov's suggestion in~\cite{Na1}, we refer to the affine Wenzl algebra and the cyclotomic Nazarov-Wenzl algebra as the \textsf{affine Brauer algebra}, and the \textsf{cyclotomic Brauer algebra}, respectively. 

When $k=1$, this algebra  reduces to the Brauer algebra originally  defined  in~\cite{Bra}. The decomposition numbers for  the Brauer algebra over $\mathbb C$  were computed in \cite{CDVM, CDVM1}, and  a conceptual explanation (up to a permutation of cell modules) in the framework of Lie theory  was given in \cite{ES, ES2}. 
 
 Throughout this paper, we always  assume $k>1$, although our arguments remain valid  when $k=1$. Let  $u$ be an indeterminate. Define   \begin{equation}
\label{w1}
W_{k,1}(u)=\sum_{a=0}^\infty\frac{\omega_a}{u^a},\end{equation}
where $\omega_a$'s are parameters within $B_{k, r}(\mathbf u_k)$ in Definition \ref{definition of Na}. The current $W_{k, 1}(u)$ corresponds to  $W_1(u)$
in \cite{AMR}. We include  $k$ as a subscript to emphasis its association with $B_{k, r}(\mathbf u_k)$, as we will later  consider 
$B_{k, r}(\mathbf u_k)$ for different values of  $k$.  

The family of parameters 
$\omega= (\omega_i)\in \mathbb C^\mathbb N$ is called \textsf{$\mathbf u_k$-admissible} (see  
\cite[Definition~3.6,~Lemma~3.8 ]{AMR})
if   \begin{equation}
\label{uadm} 
 W_{k,1}(u)+u-\frac 12 =(u-\frac{1}{2}(-1)^k)\prod_{i=1}^k \frac{u+u_i}{u-u_i}.\end{equation} 
It is proven in \cite[Theorem~5.5]{AMR} that  
$$\dim_\mathbb C B_{k, r}(\mathbf u_k)\le k^r(2r-1)!!,$$ with  equality holding   if and only if  
 $\omega$  is $\mathbf u_k$-admissible. Moreover,  
Goodman~\cite[Corollary~6.6]{G09} established that any finite-dimensional simple module $D$ of the affine Brauer algebra factors through a cyclotomic Hecke algebra if $e_1 D=0$, and through a cyclotomic Brauer algebra satisfying $\mathbf u_k$-admissible condition if $e_1D\neq 0$. Consequently, it suffices to study the representation theory of $ B_{k, r}(\mathbf u_k)$  under  the $\mathbf u_k$-admissible condition. This approach has been applied to classify finite-dimensional simple modules of $q$-analog of affine Brauer algebras over an algebraically closed field. For details, see~\cite{R}.

From this point onward, we will always assume that  $\omega$ is  $\mathbf u_k$-admissible when we discuss $B_{k, r}(\mathbf u_k)$. By \eqref{uadm}, the family of parameters $\omega$ is uniquely determined by $\mathbf u_k$ whenever $\omega$ is $\mathbf u_k$-admissible.   This explains why it is unnecessary to include $\omega$ explicitly in the notation for the cyclotomic Brauer algebra.
\begin{rem}
    In Theorem \ref{first123},  and Conjecture \ref{conjecture}, we impose  the condition 
    that $\omega_i \neq 0$ for some $0 \leq i \leq k-1$ if $r$ is even.  If this condition does not hold, then $\omega_i = 0$ for all $i \geq 0$ \cite[Lemma 3.9]{RSi}. Consequently,   by \eqref{uadm}, this is equivalent to 
    \begin{equation}
    \label{equ:conditionforu}
     (u-\frac{1}{2})\prod_{i=1}^k(u-(-1)^ku_i)   =(u-\frac{1}{2}(-1)^k)\prod_{i=1}^k(u+(-1)^ku_i).
    \end{equation}
Thus, the condition $\omega_i \neq 0$ for some $0 \leq i \leq k-1$ is equivalent to requiring that the  parameters $\mathbf{u_k}$ for $B_{k,r}((-1)^k \mathbf{u_k})$ do not satisfy the equation \eqref{equ:conditionforu}. In particular, in the level-one case (i.e., $k = 1$), this implies $u_1 \neq \frac{1}{2}$.
 \end{rem} 

\subsection{Some elementary relations for $B_{k+1, r}(\mathbf u_{k+1})$} 
We begin by  realizing $B_{k, r}(-\mathbf u_{k})$
 with arbitrary parameters  $-\mathbf u_k$  as an idempotent truncation of $B_{k+1, r} (\mathbf u_{k+1})$, where the parameter $u_{k+1}$ is chosen appropriately. For the case  $k = 1$, this result was established in the seminal work \cite{ES}, where a specific key idea played a pivotal role in the construction. In what follows, we always assume that $k\ge 2$, although our arguments remain valid when $k=1$.

Throughout this subsection, we assume that \textsf{$u_{k+1}$  is disjoint from  $\mathbf u_k$}. From  \cite[Theorem~5.5]{AMR} for   $B_{k+1, r}(\mathbf u_{k+1})$, the elements   $x_1, x_2, \ldots, x_r$   generate a  commutative subalgebra of $B_{k+1, r}(\mathbf{u}_{k+1})$. Consequently,  any finite-dimensional $B_{k+1, r}(\mathbf{u}_{k+1})$-module $M$ admits a decomposition
\begin{equation}\label{decpro}
M = \bigoplus_{\mathbf{i} \in \C^r} M_{\mathbf{i}},
\end{equation}
where $M_{\mathbf{i}} = \{m \in M \mid (x_j - i_j)^a m = 0 \text{ for all } 1 \leq j \leq r \text{ and } a\gg 0\}$.

Recall the Jucys-Murphy basis of 
$B_{k+1,r}(\mathbf u_{k+1})$ \cite[Proposition 5.8]{RSi} with each basis element being a generalized eigenvector of $x_i$, $1\le i\le r$, where  the eigenvalues lie in $I_{\mathbf u_{k+1}}:=\{\pm u_j+\Z\mid 1\le j\le k+1 \}$ \cite[Theorem 5.12]{RSi}.  
Considering the decomposition \eqref{decpro} for the regular module, it follows that there exists a system $\{e(\mathbf i)\mid \mathbf i\in I_{\mathbf u_{k+1}}^r\}$ of mutually orthogonal idempotents in $B_{k+1,r}(\mathbf u_{k+1})$ such that $e(\mathbf i)M=M_\mathbf i$ for each finite-dimensional module $M$. In fact, each $e(\mathbf i)$ lies in the commutative subalgebra generated by  $x_1, \ldots, x_r$, all but finitely many of the $e(\mathbf i)$'s are necessarily zero, and their sum is the identity element in $B_{k+1,r}(\mathbf u_{k+1})$. 
\begin{Defn}
    Suppose  $1\le i\le r$. We define \begin{itemize}\item[(1)]  $\gamma_i\in B_{k+1,r}(\mathbf u_{k+1})$, the idempotent that projects onto the generalized eigenspace of $B_{k+1, r}(\mathbf u_{k+1})$ with respect to $x_i$, where the corresponding   eigenvalue is disjoint from $u_{k+1}$,\item [(2)] $\mathbf f=\gamma_1\gamma_2\ldots \gamma_r$,    \item [(3)] $b_i=x_i+u_{k+1}$ and $ c_i=-x_i+u_{k+1}$.\end{itemize} \end{Defn}
In particular, each $\gamma_i$ and $\mf$ lie in the commutative subalgebra generated by $x_1,\ldots,x_r$.

\begin{rem}
    As explained in \cite{ES} for $B_{2, r}(\mathbf u_2)$, consider  any $B_{k+1,r}(\mathbf u_{k+1})$-module $M$. The eigenvalues of $x_i$ acting on $\mathbf fM$  are  disjoint from $u_{k+1} $. Consequently, the operators  $ b_i$ and $c_i$ act invertibly on   $\mathbf f M$,  as their actions are non-degenerate. This ensures that      $\frac{1}{b_i}$ and $\frac{1}{c_i}$ are  well-defined on $\mathbf fM$.
\end{rem}

 The following two lemmas are essentially the same as \cite[Lemma 3.12]{ES} and \cite[Proposition 3.13]{ES}, respectively. The only difference lies  choice of  $u_{k+1}$, which is disjoint from $u_1,\ldots, u_k$ rather than  $\beta$, which is disjoint from $\alpha$ in \cite{ES}. This adjustment ensures that the arguments extend naturally to the level  $k+1 $ setting considered here. 
 
 \begin{Lemma} (c.f. \cite[Lemma 3.12]{ES})
 \label{Lem:ES3.12}
 For $1\le i\le r-1$,  the following equations hold in $B_{k+1,r}(\mathbf u_{k+1})$:
 \begin{itemize}
     \item[(1)] $b_{i+1}s_i=s_ib_i-e_i+1$, and $\mathbf fs_i\frac{1}{b_i}\mathbf f= \frac{1}{b_{i+1}}\mf s_i\mf -\frac{1}{b_{i+1}}\mf e_i\frac{1}{b_i}\mf +\frac{1}{b_ib_{i+1}}\mf$,
     \item[(2)] $s_ib_{i+1}=b_is_i-e_i+1$, and $\frac{1}{b_i}\mf s_i\mf =\mf s_i\frac{1}{b_{i+1}}\mf -\frac{1}{b_i}\mf e_i\frac{1}{b_{i+1}}\mf + \frac{1}{b_ib_{i+1}}\mf $,
      \item[(3)] $c_{i+1}s_i=s_ic_i+e_i-1$, and $\mathbf fs_i\frac{1}{c_i}\mathbf f= \frac{1}{c_{i+1}}\mf s_i\mf +\frac{1}{c_{i+1}}\mf e_i\frac{1}{c_i}\mf -\frac{1}{c_ic_{i+1}}\mf$,
     \item[(4)] $s_ic_{i+1}=c_is_i+e_i-1$, and $\frac{1}{c_i}\mf s_i\mf =\mf s_i\frac{1}{c_{i+1}}\mf +\frac{1}{c_i}\mf e_i\frac{1}{c_{i+1}}\mf - \frac{1}{c_ic_{i+1}}\mf $.
 \end{itemize}
 \end{Lemma}

 \begin{Lemma}(c.f. \cite[Proposition 3.13]{ES})
  \label{Lem:ES3.13}  For $1\le i\le r-1$, the following equations hold in $B_{k+1,r}(\mathbf u_{k+1})$:  
     
\begin{multicols}{2}
    \begin{enumerate}
        \item[(1a)] $c_i\mf s_i\mf =c_is_i\mf$,
        \item[(1b)] $b_{i+1}\mf s_i\mf =b_{i+1}s_i\mf$,
         \item[(1c)] $c_i\mf e_i\mf =c_ie_i\mf$,
        \item [(1d)] $b_{i+1}\mf e_i\mf =b_{i+1}e_i\mf$,
    \end{enumerate}

    \begin{enumerate}
        \item[(2a)] $e_i\mf s_{i+1}\mf =e_is_{i+1}\mf $,
        \item [(2b)] $e_i\mf e_{i+1}\mf =e_ie_{i+1}\mf $,
        \item [(2c)] $\mf e_i\mf s_{i+1}\mf e_i\mf =\mf e_is_{i+1}e_i\mf$,
        \item [(2d)] $\mf s_i\mf s_{i+1}\mf =\mf s_is_{i+1}\mf$,
        \item [(2e)] $\mf s_i\mf e_{i+1}\mf =\mf s_ie_{i+1}\mf$, 
    \end{enumerate} 
\end{multicols}
as well as all of these equations with $i$ and $i+1$ swapped.
\end{Lemma}

The proof of  Proposition \ref{genefatingfu}(1) represents one of the main distinctions between the $ k = 1$  case discussed in \cite{ES} and the case $ k > 1$   addressed here. It seems  that the approach employed in the proof of \cite[Proposition 6.2 (iii)]{ES} could not be extended to higher levels. 
 
 \begin{Prop}
 \label{genefatingfu}  For $1\le i\le r-1$, the following equations hold in $B_{k+1,r}(\mathbf u_{k+1})$:
 \begin{multicols}{2} 
     \item[(1)] $e_i\frac{1}{b_i}e_i \mf =(1+\frac{1}{2u_{k+1}})e_i\mf $,
     \item [(2)]$e_i\frac{1}{b_i}s_i\mf =\frac{1}{2u_{k+1}}e_i\frac{1}{b_i}\mf $,
     \item [(3)] $\mf s_i\frac{1}{b_i}e_i\mf =\frac{1}{2u_{k+1}}\frac{1}{b_i}\mf e_i\mf$.
 \end{multicols}     
 \end{Prop}
 \begin{proof}
From   \cite[Lemma~4.15]{AMR},  let  
$W_i(u)=\sum_{a= 0}^\infty \frac{\omega_i^{(a)}}{u^{a}}$, 
where $\omega_i^{(a)} \in \mathbb C[x_1, x_2, \ldots, x_{i-1}]$ is a central element in $B_{k+1,i-1}(\mathbf u_{k+1})$ determined by the relations 
    $e_i x_i^a e_i=\omega_i^{(a)} e_i $, $1\le i\le r-1$.
    In particular, $W_1(u)=W_{k+1,1}(u)$
    as in \eqref{w1}.
    Moreover, by \cite[Lemma~4.15, Proposition~4.17]{AMR}, 
    $$W_{i}(u)+u-\frac{1}{2}=(W_{k+1, 1}(u)+u-\frac{1}{2})\prod_{j=1}^{i-1}\frac{(u+x_j)^2-1}{(u-x_j)^2-1}\frac{(u-x_j)^2}{(u+x_j)^2}.$$
This equation implies  that 
\begin{equation}
\label{equa:eatuk1}
 \frac{ W_{i}(u)}{u}
\mid_{u=-u_{k+1}}= -1-\frac{1}{2u_{k+1}},    
\end{equation}
because $u+u_{k+1}$ is a factor of $W_{k+1, 1}(u)+u-\frac{1}{2} $,  as given in  \eqref{uadm}.
Consequently, 
     $$ e_i\frac{1}{b_i}e_i\mf= - {W_{i}(u)}{u^{-1}}\mid _{u=-u_{k+1}}e_i\mf \overset{\eqref{equa:eatuk1}}= (1+\frac{1}{2u_{k+1}}) e_i\mf ,$$
     and (1) follows. 
 We have 
   \begin{align*}
     e_i\frac{1}{b_i}s_i\mf &= 
    e_i s_i\frac{1}{b_{i+1}}\mf -e_i\frac{1}{b_i}e_i\frac{1}{b_{i+1}}\mf + e_i\frac{1}{b_ib_{i+1}}\mf\ \   \text{ by Lemma \ref{Lem:ES3.12}(2)} \\
  &  \overset{(1)}= 
  e_is_i\frac{1}{b_{i+1}}\mf -(1+\frac{1}{2u_{k+1}})e_i\frac{1}{b_{i+1}}\mf +e_i\frac{1}{b_ib_{i+1}}\mf \\
  & = e_i\frac{1}{b_{i+1}}\mf -(1+\frac{1}{2u_{k+1}})e_i\frac{1}{b_{i+1}}\mf +e_i\frac{1}{b_ib_{i+1}}\mf
     \\
     &=
     e_i(\frac{2u_{k+1}-b_i}{2u_{k+1}b_ib_{i+1}})\mf
  =   e_i(\frac{c_i}{2u_{k+1}b_ib_{i+1}})\mf\\
  &=
  e_i(\frac{b_{i+1}}{2u_{k+1}b_ib_{i+1}})\mf=
  \frac{1}{2u_{k+1}} e_i \frac{1}{b_i}\mf, 
   \end{align*}
proving (2).  
We leave  (3) to the reader, as it can be derived similarly from (1).  \end{proof}

\subsection{The algebra isomorphism}
The following definition is motivated by \cite{ES}, where the $k=1$ was discussed. 
\begin{Defn}
\label{def-of-tile}
    For all  $1\le i\le r-1$, and $1\le j\le r$ define 
    \begin{multicols}
        {2} \item [(1)]$Q_i=\sqrt{\frac{b_{i+1}}{b_i}}\mf$,         
        \item [(2)] $\tilde s_i= -Q_is_i Q_i +\frac{1}{b_i}\mf$, 
        \item [(3)] $\tilde e_i= Q_i e_iQ_i$, \item [(4)] $\tilde x_j=-x_j \mf$.        
        \end{multicols}
   
\end{Defn}
These are elements in $\mf B_{k+1, r}(\mathbf u)\mf $. Naturally, as in \cite[(4.1)]{ES}, we must fix the square root of $\frac{b_{i+1}}{b_i}$ once and for all.   
Define 
\begin{equation}
\label{hatW}  \widehat W_1(u)+u-\frac12 = (u-\frac{1}{2}(-1)^k) \prod_{i=1}^k \frac{u-u_i}{u+u_i}.\end{equation}
Since we assume that $\omega$ (for the cyclotomic Brauer algebra $B_{k,r}(-\mathbf u_k)$) is 
$-\mathbf u_{k}$-admissible, by  \eqref{w1}--\eqref{uadm}, we have 
$\widehat W_1(u)=\sum_{i\ge 0}\omega_i u^{-i}$ with the property that 
$ e_1 x_1 ^i e_1=\omega_i e_1$ in 
$ B_{k,r}(-\mathbf u_k)$,
for all  $i\in \mathbb N$. In other words, 
  $\widehat W_1(u)=W_{k,1}(u)$ for $B_{k,r}(-\mathbf u_k)$ as defined in \eqref{uadm}.

 The proof of the following result using the generating function $\widehat W_1(u)$ is the most surprising aspect for higher levels, as the approach for $k=1$ in \cite{ES} does not seem  applicable in this case.   
\begin{Lemma}
\label{lem:rela5}
For any $i\in \mathbb N$,   $\tilde e_1 \tilde x_1^i\tilde e _1= \omega_i \tilde e_1$ in $\mf B_{k+1, r}(\mathbf u_{k+1})\mf $.        
\end{Lemma}
\begin{proof}
  By  Definition~\ref{def-of-tile}, Definition~\ref{definition of Na}(17), and \eqref{hatW},  we have 
  \begin{align*}
    \tilde e_1 \sum_{a\ge 0}\frac{(-x_1)^a}{u^a}\tilde e_1&= Q_1 e_1 \frac{b_2}{b_1}\sum_{a\ge 0}\frac{(-x_1)^a}{u^a}e_1 Q_1\\
    &= Q_1 e_1 \frac{c_1}{b_1}\sum_{a\ge 0}\frac{(-x_1)^a}{u^a}e_1 Q_1 \\
    &=Q_1 e_1 \frac{u_{k+1}-x_1}{u_{k+1}+x_1} \frac{u}{u+x_1} e_1 Q_1\\
    &= Q_1 e_1( \frac{2u_{k+1}}{u_{k+1}+x_1} \frac{u}{u+x_1}-\frac{u}{u+x_1} )e_1 Q_1\\
    &= Q_1 e_1(\frac{2u_{k+1}u}{u-u_{k+1}}( \frac{1}{u_{k+1}+x_1}-\frac{1}{u+x_1})-\frac{u}{u+x_1} )e_1 Q_1\\
     &= \frac{2u_{k+1}u}{u-u_{k+1}}Q_1  e_1 \frac{1}{u_{k+1}+x_1} e_1 Q_1 - \frac{ u+u_{k+1}}{u-u_{k+1}}Q_1  e_1 \frac{u}{u+x_1} e_1 Q_1\\
     & = \frac{2u_{k+1}u}{u-u_{k+1}}
     (1+\frac{1}{2u_{k+1}})\tilde e_1 -\frac{ u+u_{k+1}}{u-u_{k+1}}    W_{k+1,1}(-u)\tilde e_1 \ \ \  \text{ by Proposition~\ref{genefatingfu}}(1) \\
     &\overset{\eqref{uadm}}= \frac{2u_{k+1}u}{u-u_{k+1}}
     (1+\frac{1}{2u_{k+1}})\tilde e_1 -(u+\frac{1}{2})\frac{u+u_{k+1}}{u-u_{k+1}}\tilde e_1
     +  
     (u-\frac{1}{2}(-1)^k)\prod_{i=1}^k\frac{u-u_i}{u+u_i}\tilde e_1\\
     &= \widehat W_1(u)\tilde e_1= W_{k,1}(u)\tilde e_1.
  \end{align*} 
 This immediately leads  to the desired  result. 
\end{proof}

\begin{Lemma}
\label{Lem:12-14}For any $1\le i\le r-1$, the following equations hold in $\mf B_{k+1,r}(\mathbf u_{k+1})\mf $:
\begin{multicols}{2}
    \item[(1)] $ \tilde e_i(\tilde x_i+\tilde x_{i+1})= (\tilde x_i+\tilde x_{i+1})\tilde e_i=0$,
    \item[(2)] $ \tilde e_i\tilde x_l=\tilde x_l\tilde e_i$,  $ \tilde s_i\tilde x_l=\tilde x_l\tilde s_i$  if $l\neq i,i+1 $,
    \item[(3)] $ \tilde x_i\tilde x_j=\tilde x_j\tilde x_i$,
    \item[(4)] $(\tilde x_1+u_1)(\tilde x_1+u_2)\cdots (\tilde x_1+u_k)\mf=0 $.\end{multicols}

\end{Lemma}
\begin{proof} Since $Q_i$ commutates with $x_j$ and $Q_i \mf=\mathbf f Q_i=Q_i$, we have 
$$     \tilde e_i( \tilde x_i+\tilde x_{i+1}) = - Q_i e_i Q_i(x_i+x_{i+1})\mf 
       = - Q_i e_i (x_i+x_{i+1})Q_i=0, $$
where the last equality follows from Definition \ref{definition of Na}(17).
Hence, the first equality in (1) follows.    The second identity in (1) can be derived similarly. Next, we compute
   $$  \tilde e_i \tilde x_l= -Q_i e_i Q_i x_l\mf = -Q_i e_ix_l  Q_i \mf = -Q_i x_l e_i Q_i \mf
    = -x_l Q_ie_i Q_i=\tilde x_l \tilde e_i.
   $$
   Here the third equality follows from Definition~\ref{definition of Na}(8). 
   One can derive the second one in (2)  similarly.
We remark that 
   (3) follows from the definition directly. Finally, we have 
   $$(  x_1-u_1)(  x_1-u_2)\cdots (  x_1-u_k)(x-u_{k+1})=0 $$
   in $B_{k+1,r}(\mathbf u_{k+1})$ by Definition \ref{definition of Na} (18). Since $ (x_1-u_{k+1})\mf $ is invertible, we immediately have   
   $$(  x_1-u_1)(  x_1-u_2)\cdots (  x_1-u_k)\mf =0,$$ proving (4). 
\end{proof}

\begin{Lemma}
\label{Lem:relaforBrauer}
The following relations hold in $\mf B_{k+1,r}(\mathbf u_{k+1})\mf $:
\begin{multicols}{2}
\begin{itemize}
    \item  [(1)] $\tilde s_i^2=\mf $,  $1\le i< r$,
\item[(2)] $\tilde s_i\tilde s_j=\tilde s_j\tilde s_i$, if $|i-j|>1$,
\item[(3)] $\tilde s_i\tilde s_{i+1}\tilde s_i\!=\!\tilde s_{i+1}\tilde s_i\tilde s_{i+1}$, $1\!\le\! i\!<\!r\!-\!1$,
\item[(4)]$\tilde e_i \tilde s_i=\tilde e_i=\tilde s_i\tilde e_i$, $1\leq i\leq r-1$,
\item[(5)] $\tilde s_i\tilde e_j=\tilde e_j\tilde s_i$, if  $|i-j|>1$,
\item[(6)] $\tilde e_i\tilde e_j=\tilde e_j\tilde e_i$, if  $|i-j|>1$,
\item[(7)]$\tilde s_i\tilde e_{i+1}\tilde e_i=\tilde s_{i+1}\tilde e_i$, $1\leq i\leq r-2$,
\item [(8)]$\tilde e_i\tilde  e_{i+1}\tilde s_i =\tilde e_i \tilde s_{i+1}$, $1\leq i\leq r-2$,
\item[(9)]$\tilde e_i \tilde e_{i+1}\tilde e_i =\tilde e_{i+1}$, $1\leq i\leq r-2$,
\item[(10)]$\tilde e_{i+1} \tilde e_{i}\tilde e_{i+1} =\tilde e_{i}$, $1\leq i\leq r-2$.
\end{itemize}
\end{multicols}
 \end{Lemma}
\begin{proof}
    Note that the corresponding proof for $k=1$ in  the proof of Lemmas 6.4 -- 6.11 of \cite{ES}
      relies only on Lemma 3.2, and Proposition 3.13, and Proposition 6.2 in \cite{ES}. 
      The analogs of  these results  for $k>1$ have been established in Lemma \ref{Lem:ES3.12}, Lemma \ref{Lem:ES3.13}, and Proposition \ref{genefatingfu}, respectively. Therefore, the result follows by applying the same arguments used in \cite{ES} for $k=1$ to the cases $k>1$.   
\end{proof}

\begin{Lemma}
\label{Lem:rel1516} For any $1\le i\le r-1$, the following equations hold in $\mf B_{k+1,r}(\mathbf u_{k+1})\mf $ :
\begin{multicols}
    {2}

  \item[(1)]$\tilde s_i \tilde x_i- \tilde x_{i+1}\tilde s_i=\tilde e_i-\mf $,
\item[(2)]$ \tilde x_i \tilde s_i-\tilde s_i  \tilde x_{i+1}=\tilde e_i-\mf  $.
\end{multicols}
  \end{Lemma}
  \begin{proof}
 
   We have 
\begin{align*}
    \tilde s_i \tilde x_i\tilde s_i +\tilde s_i-\tilde e_i&= -\tilde s_i   x_i\tilde s_i +\tilde s_i-\tilde e_i\\
&= Q_i s_i Q_i x_i \tilde s_i + \tilde s_i -\tilde e_i -\frac{1}{b_i}\mf x_i\tilde s_i \\
&= Q_i(x_{i+1}s_i+e_i-1)Q_i\tilde s_i+ \tilde s_i -\tilde e_i -\frac{1}{b_i}\mf x_i\tilde s_i \  \ \text{ by Definition~\ref{definition of Na}(15)}  \\
&= x_{i+1} Q_i s_i Q_i \tilde s_i + \tilde e_i\tilde s_i - \frac{b_{i+1}}{b_i}\mf \tilde s_i+ \tilde s_i -\tilde e_i -\frac{1}{b_i}\mf x_i\tilde s_i\\
&= x_{i+1} Q_i s_i Q_i \tilde s_i  - \frac{b_{i+1}}{b_i}\mf \tilde s_i+ \tilde s_i  -\frac{1}{b_i}\mf x_i\tilde s_i
 \ \ \text{ by Lemma~\ref{Lem:relaforBrauer}(4)} \\
&=  -x_{i+1}\tilde s_i \tilde s_i + \frac{x_{i+1}}{b_i}\mf \tilde s_i - \frac{b_{i+1}}{b_i}\mf \tilde s_i+ \tilde s_i  -\frac{1}{b_i}\mf x_i\tilde s_i \\
&= -x_{i+1}\mf + \frac{x_{i+1}}{b_i}\mf \tilde s_i - \frac{b_{i+1}}{b_i}\mf \tilde s_i+ \tilde s_i  -\frac{1}{b_i}\mf x_i\tilde s_i\ \  \text{by  Lemma~\ref{Lem:relaforBrauer}(1)} \\
&=  -x_{i+1}\mf + \frac{x_{i+1}}{b_i}\mf \tilde s_i - \frac{x_{i+1}+u_{k+1}}{b_i}\mf \tilde s_i+ \tilde s_i  -\frac{1}{b_i}\mf x_i\tilde s_i\\
&=-x_{i+1}\mf=\tilde x_{i+1}.
\end{align*}
By Lemma~\ref{Lem:relaforBrauer}(1)(4), we immediately have  (1) and (2).  
\end{proof}

We use standard terminology on $a$-compositions,  $a$-partitions, and related concepts  from  \cite{Ma}. Let $\Lambda^+_a(r)$ denote the set of all $a$-partitions of $r$. If $\lambda$ and $\mu$ are two $a$-partitions we say that $\lambda$ is obtained from $\mu$ by
adding a box if there exists a pair $(i, s)$ such that $\lambda^{(s)}_i=\mu^{(s)}_i+1$
and $\lambda^{(t)}_
j =
\mu^{(t)}_
j$ for $(j, t) \neq (i, s)$. In this case, we  also say that $\mu$ is obtained
from $\lambda$ by removing a box, and we write $\lambda\setminus\mu =p$, where  $p=(i, \lambda^{(s)}_i, s)$.
Furthermore,  the triple $(i, \lambda^{(s)}_i, s)$ is referred to as  an addable node of $\mu$ and a
removable node of $\lambda$.

Fix an integer $f$ with $0\le f\leq \lfloor r/2\rfloor$ and let $\lambda$ be an $a$-partition of $r-2f$. An $r$-updown $\lambda$-tableau, or  simply an updown $\lambda$-tableau $\s$,
is a sequence $\s = (\s_1, \s_2, \ldots, \s_r)$  of $a$-partitions where $\s_r =\lambda$,  and
the $a$-partition $\s_i$ is obtained from $\s_{i-1}$ by either adding or removing
a box, for $i = 1, 2, \ldots, r$, where we set $\s_0$  to the empty $a$-partition. Let $\mathscr T^{ud}_{a,r}(\lambda)$  be the set of updown $\lambda$-tableaux of $r$. Note that $\lambda$ is
an $a$-partition of $r-2f$,  each element of $\mathscr T^{ud}_{a,r}(\lambda)$ 
is an $r$-tuple of
$a$-partitions, so the index $r$ is necessary for this notation.

\begin{Theorem}
\label{sioiemd}
    There exits an algebra isomorphism $ \phi: B_{k,r}(-\mathbf u_k) \rightarrow \mf B_{k+1,r}(\mathbf u_{k+1})\mf $ that maps 
     $s_i, e_i, x_j$ to $\tilde s_i, \tilde e_i, \tilde x_j$, respectively,  for all $1\le i\le r-1$ and $1\le j\le r$. 
\end{Theorem}

\begin{proof}
We first demonstrate that  $\phi$  is a well-defined algebra homomorphism. The relation (5) in Definition \ref{definition of Na} is satisfied by Lemma \ref{lem:rela5}. Similarly, the relations (8)-(9), (17), and the cyclotomic relation  $ (x_1+u_1)(x_1+u_2)\ldots (x_1+u_{k})=0$ in $B_{k,r}(-\mathbf u_k)$ hold  by Lemma \ref{Lem:12-14}.
The relations (10) and (11) are verified using Lemma \ref{Lem:rel1516}. Lastly, the remaining relations for Brauer algebra are established by Lemma \ref{Lem:relaforBrauer}. This concludes the proof that  $\phi $  is well-defined.
   
    Next, we observe that  $\phi$  is surjective. This follows from arguments similar to the case  $k = 1 $ in the proof of \cite[Theorem 4.3]{ES}, with Lemma \ref{Lem:ES3.13} replacing \cite[Proposition 3.13]{ES}.
    
Finally, to prove the injectivity of    $\phi$,  it suffices to demonstrate  that the dimension of $ B_{k,r}(-\mathbf{u}_k)$ coincides with  the dimension of   $\text{im} \phi$. This can be verified by analyzing the generalized eigenvalues of the Jucys-Murphy basis of  $B_{k+1,r}(\mathbf{u}_{k+1})$.  See \cite[\S5]{RSi}, where the Jucys-Murphy basis for any cell module for 
$B_{k+1,r}(\mathbf{u}_{k+1})$ 
is established. 
Specifically, $B_{k+1,r}(\mathbf u_{k+1})$ has  the Jucys-Murphy basis 
\[\{m_{\s,\t} \mid \s,\t \in \mathscr {T}^{up}_{k+1,r}(\lambda), (f,\lambda)\in \Lambda_{k+1,r}\}  \]
 given in \cite[Proposition 5.8]{RSi}. We can embed   $\mathscr {T}^{up}_{k,r}(\lambda) $  into 
$\mathscr {T}^{up}_{k+1,r}(\lambda) $ via the identification sending   $\lambda\in \Lambda_{k,r}$ to  $(\lambda, \emptyset)\in \Lambda_{k+1,r}$. 
By the definition of $\mf$ and  \cite[Theorem 5.12]{RSi}, we have 
$$ \mf m_{\s,\t}\mf =\begin{cases}
     m_{\s,\t} & \text{if  $\s,\t\in \mathscr {T}^{up}_{k,r}(\lambda)$,}\\
    0  & \text{otherwise.}
\end{cases} $$
This implies 
$$\dim_\mathbb C \mf{B}_{k+1, r}(\mathbf{u}_{k+1}) \mf =\sum_{\lambda\in  {\Lambda}_{k, r}} |\mathscr {T}^{up}_{k,r}(\lambda)|^2,$$   
which matches the dimension  of  $B_{k,r}(-\mathbf{u}_k)$, completing the proof.
\end{proof}

``\textbf{Proof of Theorem~A:} "
    Under the  assumption, we iteratively  apply Theorem \ref{sioiemd}   $k$ times to obtain the result. More specifically, if  we denote $\mathbf f$ in Theorem~\ref{sioiemd} by $\mathbf f^{k+1}$, then the required idempotent $e$ is $\mathbf f^{k+1} \mathbf f^{k+2}\cdots\mathbf f^{2k}$.  Moreover, we have $e=\gamma_1'\ldots \gamma_r'$, where each $\gamma_i'$ is the idempotent projecting  onto the generalized eigenspace corresponding  to $x_i$, $1\le i\le r$,  with  the associated  eigenvalues disjoint from $u_{j}$ for all $j$,  $k+1\le j\le 2k$.  
\qed

\section{  Endomorphism algebra in parabolic category $\mathcal O^\frp$ of type $ D_n$}

This section introduces the parabolic category $\mathcal O^\frp $ associated with a parabolic subalgebra $\mathfrak p$ of the orthogonal Lie algebra $\mathfrak{so}_{2n}$. The primary object is  to 
realize $B_{2k,r}(\mathbf u_{2k})$ as an endomorphism algebra within the parabolic category $\mathcal O^\frp$ of type $ D_n$, for any chosen parameters $(-1)^k\mathbf u_k$. This realization  establishes  a  connection between the cyclotomic Brauer algebra $B_{k, r}((-1)^k \mathbf u_k)$ and  the  parabolic category $\mathcal O^\frp$ in type $D_n$. 
 
\subsection{The Lie algebra of type $D_n$} 
Throughout,  let $V$ be the $2n$-dimensional  vector space over $\mathbb C$  with basis 
\begin{equation}\label{basis} \{v_i\mid i\in [n]\},\end{equation}  where $[n]:=\{1,2,\ldots, n\}\cup\{-1, -2, \ldots, -n\}$.
      Let $ (\ ,\ ): V\otimes V\rightarrow \mathbb C$ denote  the non-degenerate symmetric bilinear form such that 
      $(v_i, v_j)=\delta_{i, -j}$.
  The Lie algebra  $\mathfrak g$ of type $D_n$ is   
  \begin{equation}\label{defofg}
\mathfrak {so}_{2n}=\{g\in \End(V) \mid (gx,y)+(x,gy)=0 \text{ for all $x,y\in V$}\}.
\end{equation}
For all $i, j\in  [n]$, let $E_{i, j}$ be the usual matrix unit with respect to the basis of $V$ in \eqref{basis}. Then   $\mathfrak{so}_{2n}$ has basis \begin{equation}\label{sobas} \{F_{i, i}\mid 1\le i\le n\}\cup \{F_{\pm i, \pm j}\mid 1\le i<j\le n\}, \end{equation} 
where
$ F_{i,j}=E_{i,j}- E_{-j,-i}$ (see \cite[(7.9)]{Mol}). There is a standard triangular decomposition  \begin{equation} \label{trian}\mathfrak g=\mathfrak n^{-}\oplus \mathfrak h\oplus \mathfrak n^+,\end{equation} where  $\mathfrak h:=   \bigoplus_{i=1}^n \mathbb C h_i$ is the standard  Cartan subalgebra  with   $h_i=F_{i,i}$,  $1\leq i\leq n$, 
 and  $\mathfrak n^+$ has   basis $
\{F_{ i,\pm j}\mid 1\le i<j\le n\}$. 
Let $\mathfrak h^*$ be the linear dual  of $\mathfrak h$ with the dual basis $\{\epsilon_i\mid 1\le i\le n\}$ such that  $\epsilon_i(h_j)=\delta_{i, j}$  for all  $1\le i, j\le n$.
Let $\Phi$ denote the root system associated with  \eqref{trian}. 
Then $\Phi=\Phi^+\cup -\Phi^+$, where  the set of positive roots is 
\begin{equation}\label{posro} \Phi^+=\{\epsilon_i\pm \epsilon_j\mid 1\le i<j\le n\}.\end{equation}  The set of 
simple roots is $\Pi=\{\alpha_1, \alpha_2, \ldots, \alpha_{n-1}, \alpha_n\}$, where 
 \begin{equation}\label{aln} \alpha_i=\epsilon_i-\epsilon_{i+1},  1\le i\le n-1, \ \text{and} \ \alpha_n=
\epsilon_{n-1}+\epsilon_n   . 
\end{equation}

\subsection{Parabolic category $\mathcal O^\frp$}
Throughout, we fix the subset $I$ of $\Pi$ as in \eqref{lever}.
It is the $I_1$ for $\mathfrak{so}_{2n}$ in \cite{GRX, RS-cyc}.
\begin{Defn}\label{defofpi1}
Fix positive integers $q_1,q_2,\ldots,q_k$ such that $q_t=p_t-p_{t-1}$, $1\le t\le k$, and define 
 $\mathbf p_j=\{p_{j-1}+1, p_{j-1}+2,\ldots, p_j\}$, for all $1\le j\le k$.
\end{Defn}

There is a root system $\Phi_{I}= \Phi\cap \mathbb Z I$ with  $\Phi_{I}^+=\Phi^+\cap \mathbb Z I$ as the set of positive roots. Let $\mathfrak p$ be the standard parabolic subalgebra of $\mathfrak {so}_{2n}$  with respect to  $I$ defined as in~\eqref{lever}. Then  $\mathfrak p=\mathfrak u\oplus \mathfrak l$, where 
$\mathfrak u$ is the nilradical of $\mathfrak p$, and $\mathfrak l$ is the Levi subalgebra. Let $W_I$ be the Weyl group associated with $\mathfrak l$.

Recall that the BGG  category  $\mathcal O$ is the category of finitely generated $\mathfrak {so}_{2n}$-modules which are locally finite over $\mathfrak{n}^+$ and semi-simple over $\mathfrak h$.
Let $\mathcal{O}^{\mathfrak p}$ be  the full subcategory of $\mathcal O$  consisting of all $\mathfrak{so}_{2n}$-modules which are locally   $\mathfrak p$-finite. 

For any  $\lambda\in\Lambda^\mathfrak p$, where $\Lambda^\frp$ is defined as in \eqref{pdom}, there exists a unique irreducible $\mathfrak l$-module $L_I( \lambda)$, which can be viewed as a $\mathfrak p$-module by defining the action of $\mathfrak{u}$ to be trivial. 
The corresponding   \textit{parabolic Verma module} is  
$$
M^\mathfrak p(\lambda)=\mathbf U(\mathfrak g)\otimes_{\mathbf U(\mathfrak p)} L_I( \lambda).$$

\subsection{Certain simple parabolic Verma modules}
In this subsection, we revisit  a simplicity criterion for parabolic Verma modules as presented in  \cite{XZ},  which serves as  a refinement of Jantzen's criterion  as given in    \cite[Theorem 9.13]{Hum}. This refinement enables us to identify a suitable tilting module in $\mathcal O^\frp $, such that the corresponding endomorphism algebra is isomorphic to $B_{2k, r}(\mathbf u_{2k})$.

For any $\lambda\in \Lambda^\mathfrak p$, define 
\begin{equation}\label{psipl} \begin{aligned}
\Psi^+_\lambda&= \{ \beta\in \Phi^+\setminus\Phi_I\mid \langle \lambda+\rho, \beta^\vee\rangle \in \Z_{>0} \},\\
\Psi^{++}_\lambda &= \{\beta\in \Psi^+_\lambda \mid \langle s_\beta(\lambda+\rho), \alpha^\vee \rangle\neq 0 \text{ for all } \alpha\in \Phi_I 
\},\\
\end{aligned}
\end{equation}
where   
\begin{equation}
\label{equ:defofrho}
  \rho=\sum_{i=1}^n (n-i)\varepsilon_i
\end{equation}  is half the  sum of all  positive roots, and 
$s_\beta$ is the reflection corresponding to $\beta$.

A weight $\mu$ is called \textsf{ $\Phi_I$-regular} if 
$\langle \mu, \alpha^\vee\rangle \neq 0$ for all $\alpha\in \Phi_I$. Therefore, $\Psi^{++}_\lambda$ consists of all elements $\beta\in \Psi_\lambda^+$ such that  $s_\beta(\lambda+\rho)$  are  $\Phi_I$-regular. 
The following simplicity criterion is a special case of \cite[Corollary 5.15]{XZ}.
\begin{Prop}
\label{Pro:simcri}\cite[Corollary 5.15]{XZ}   For any $\lambda\in \Lambda^\mathfrak p$, $M^\mathfrak p(\lambda)$ is simple 
   if and only if $ \Psi^{++}_\lambda$
   contains only roots of form
   $\epsilon_i+\epsilon_j$ for $i, j\in \mathbf p_s$ with $i<j$, where  $\mathbf p_s$ is   defined as  in Definition~\ref{defofpi1}.   
 Moreover, $(\lambda+\rho)_j= (\lambda+\rho)_n=0$
   and $(\lambda+\rho)_l\neq -(\lambda+\rho)_i$ for all $l\in \mathbf p_s$. In particular, $M^\mathfrak p(\lambda)$ is simple, and hence is also tilting if  $\Psi^{++}_\lambda=\emptyset$.    
\end{Prop}

\begin{proof}
    It follows from \cite[\S~11.8]{Hum} that $M^\mathfrak p(\lambda)$ is tilting if and only if it is simple. Therefore, the last statement follows from the first one,  which was established  in \cite[Corollary 5.15]{XZ}. \end{proof}

 For any $\mathbf c=(c_1, c_2, \ldots, c_k)\in  \mathbb C^k$, let
$\lambda_\mathbf c$ be defined as in \eqref{deltac}.  
From this point on,  we may identify $\lambda=\sum_{i=1}^n\lambda_i\epsilon_i$ with $(\lambda_1,\ldots, \lambda_n)$. 
Similarly, for $\lambda_\mathbf c+\rho$ we may identify it with $(\mathbf d_1, \ldots, \mathbf d_k)$ by \eqref{equ:defofrho}, where 
\begin{equation}\label{dt}  (\mathbf d_s)=
(  u_s-\frac{1}{2},   u_s-\frac{3}{2}, \ldots,   u_s+\frac{1}{2}-q_s ), \text{ for } 1\le s\le k.\end{equation} 
 We call $\mathbf d_s$ the $s$-th part of $\lambda_\mathbf c+\rho$. 
 Thus,   $u_1, u_2, \ldots, u_k$ satisfy \eqref{cjuj}.

\begin{Lemma}\label{posi11} For any $\mathbf u_k\in \mathbb C^k$, there exists  a
 $\mathbf c=(c_1, c_2, \ldots, c_k)\in \mathbb C^k$ along with      positive integers $q_1,q_2, \ldots, q_k$ satisfying   $\min\{q_1, q_2, \ldots, q_k\}\gg 0$, such that  for all   
$\beta=\epsilon_i+\epsilon_j\in  \Psi_{\lambda_\mathbf c}^+$, we have $\beta\notin \Psi^{++}_{\lambda_{\mathbf c}}$ 
\end{Lemma}
\begin{proof} For any fixed $\mathbf u_k $, we can pick $\mathbf c$ along with positive integers $q_1, q_2, \ldots, q_k$ such that   $\min\{q_1, q_2, \ldots, q_k\}\gg 0$, and 
\eqref{cjuj} holds.

We assume $i<j$  without loss of generality, where $\beta=\epsilon_i+\epsilon_j$. Then, either $i, j\in \mathbf p_{t}$  or $i\in \mathbf p_t$  and $j\in \mathbf p_s $  for some $1\le t<s\le k$, where $\mathbf p_t$ is defined as in Definition~\ref{defofpi1}. We define \begin{equation} \label{iprime} i'=i-p_{t-1} \text{ if $i\in \mathbf p_t$. }\end{equation}
Suppose  $i, j\in \mathbf p_t$. If $\Psi_{\lambda_{\mathbf c}}^+=\emptyset$, there is nothing to be proved. Otherwise, 
we have 
\begin{equation}\label{ine12}  \langle\lambda_\mathbf c+\rho, \beta^\vee \rangle = (  u_t-i'+\frac{1}{2})+(  u_t-j'+\frac{1}{2}) \in \mathbb Z_{>0}. \end{equation}
This implies $  u_t-i'+\frac{1}{2}>0$ and $2u_t\in \mathbb Z$.
Since  $q_i\gg0$ for all $1\le i\le k$,  it follows that $q_t>2u_t-i'+1$, and hence 
$p_t>l$, where $l':=2  u_t-i'+1$.  Moreover, by \eqref{ine12} we have  $j'<l'$.  This ensures  that $l\in \mathbf p_t$. 
Thus, 
$$\langle s_\beta(\lambda_\mathbf c+\rho),\epsilon_j-\epsilon_l\rangle=
-(  u_t-i'+\frac{1}{2} +   u_t-l'+\frac{1}{2})=-(2  u_t-i'+1-l')=0, $$
    forcing  $\beta\notin \Psi^{++}_{\lambda_{\mathbf c}}$.

Suppose $i\in \mathbf p_t$  and $j\in \mathbf p_s $  such that  $1\le t<s\le k$. 
For any 
$\beta\in \Psi_{\lambda_\mathbf c}^+$, we have 
\begin{equation}
\label{Equ:betacse2}
    \langle\lambda_\mathbf c+\rho, \beta^\vee\rangle = (  u_t-i'+\frac{1}{2})+(  u_s-j'+\frac{1}{2})=  u_t+  u_s-i'-j'+1 \in \mathbb Z_{>0}.
\end{equation}  
   This implies $  u_t+  u_s\in \mathbb Z$, 
   and  that either $  u_t-i'+\frac{1}{2}>0$ or  $  u_s-j'+\frac{1}{2}>0$.
   Without loss of  generality, we assume 
   $u_t-i'+\frac{1}{2}>0$.
   
   Next,  we compute 
   $s_\beta(\lambda_\mathbf c+\rho)$, which is obtained from $\lambda_\mathbf c+\rho$
   by replacing $(\mathbf d_t)$ and $(\mathbf d_s)$ with $\mathbf i$ and $\mathbf j$, respectively, where
   $$ \begin{aligned} \mathbf i&=(  u_t-\frac{1}{2}, \ldots,   u_t-i'+\frac{3}{2}, { -(  u_s+\frac{1}{2}-j')}, \ldots,   u_t+\frac{1}{2}-q_t), \\
    \mathbf j&= (  u_s-\frac{1}{2}, \ldots,   u_s-j'+\frac{3}{2}, {-(  u_t+\frac{1}{2}-i')}, \ldots,   u_s+\frac{1}{2}-q_s). \\
    \end{aligned}$$
   Using \eqref{Equ:betacse2}, and noting that $q_t\gg 0$, we derive:  
    \[\begin{aligned}  (  u_t+\frac{1}{2}-i'-1) -({-(  u_s+\frac{1}{2}-j' )} ) &=  u_t+  u_s-i'-j'\ge 0,\\
       (  u_t+\frac{1}{2}-i'-1)-(  u_t+\frac{1}{2}-q_t)&=q_t-i'-1>   u_t+  u_s
+1-i'-j'>0,  \\
-(  u_s+\frac{1}{2}-j' ) -(  u_t+\frac 12-q_t)& =q_t-(  u_t+  u_s+1-j')>0.
\\
\end{aligned}  \]
Since $  u_s+  u_t\in \mathbb Z$,  the above inequalities implies that  $-(  u_s+\frac12-j')=  u_t-l'+\frac 12$ for some $i+1\le l<p_t$. Thus, $\epsilon_{i}-\epsilon_l\in \Phi_I$ and $\langle s_\beta(\lambda_\mathbf c+\rho),\epsilon_{i}-\epsilon_l\rangle=0$. This forces
 $\beta\notin \Psi^{++}_{\lambda_{\mathbf c}}$.
\end{proof}
It follows from the proof of Lemma~\ref{posi11} that $\beta\not\in \Psi_{\lambda_\mathbf c}^{++}$ if  $\min\{q_1, q_2, \ldots, q_k\}\gg 0$, and $\beta=\epsilon_i+\epsilon_j\in \Psi^+_{\lambda_{\mathbf c}}$. However, we require additional restrictions on $q_1, q_2, \ldots, q_k$ when we prove the following result. Notably,  the condition for $q_1, q_2, \ldots, q_k$  in the result below is compatible with that in  Lemma~\ref{posi11}.

\begin{Lemma}
\label{pos12} For any  $\mathbf u_k\in \mathbb C^k$, there exists a 
$\mathbf c:=(c_1, c_2, \ldots, c_k)\in \mathbb C^k$  along with  associated positive integers $q_1, q_2, \ldots, q_k$ satisfying   $\min\{q_1, q_2, \ldots, q_k\}\gg 0$,   such that for all   
$\beta=\epsilon_i-\epsilon_j\in  \Psi_{\lambda_\mathbf c}^+$, we have $\beta\notin \Psi^{++}_{\lambda_{\mathbf c}}$.  
\end{Lemma}

\begin{proof}  
We can pick $\mathbf c$ with the associated $q_1, q_2, \ldots, q_k$ such that \eqref{cjuj} holds, and    $\min\{q_1, q_2, \ldots, q_k\}\gg 0$.
 We can assume    $ i \in \mathbf  p_{s}$  and $j\in \mathbf p_{t}$  with $1\le s<t\le k$ since 
$\beta=\epsilon_i-\epsilon_j\in \Psi_{\lambda_{\mathbf c}}^+$. 
Thus, 
\begin{equation}
\label{Equ:betai-j}
 \langle \lambda_\mathbf c+\rho, \beta^\vee\rangle=  u_s-  u_t+j'-i'\in \mathbb Z_{>0},   
\end{equation}
where $ i'$ is defined as in \eqref{iprime}.

We write $  u_a-  u_b=m_{ab}$ and,   by \eqref{Equ:betai-j},   assume without loss of generality that $m_{ab}\in \Z$ for all $a,b$ . By appropriately rearranging $  u_1,   u_2, \ldots,   u_k$,  we may assume that 
\begin{equation} \label{mine} m_{12}\ge m_{13}\ge \ldots \ge m_{1k}.\end{equation} 
For any $1<s<t$, we have  $  u_s-  u_t=m_{1t}-m_{1s}\leq 0$, which implies 
\begin{equation}
\label{Equ:condusut}
    u_s-\frac{1}{2} \le   u_t-\frac{1}{2}.  
\end{equation}
  We choose $q_t$ and $q_s$ such that 
  \begin{equation}
  \label{Equ:conqtqs}
      q_t\ge q_s-m_{st}.
  \end{equation}
This selection is feasible since we can first pick  sufficiently large $p_1, p_2, \ldots, p_k\gg 0$, independently,  and then determine $c_1, c_2, \ldots, c_k$ using \eqref{cjuj}.  From \eqref{Equ:conqtqs}, it follows  that 
  $$(   u_s+ \frac{1}{2}-q_s) - (  u_t+\frac{1}{2}-q_t )= q_t-q_s+m_{st}\ge 0.$$
  As a result,  
 $ (\mathbf d_s)\subset (\mathbf d_t) $ where $(\mathbf d_t)$ is defined as in \eqref{dt}. 
  This inclusion ensures  that  the  $t$-th part of $s_\beta (\lambda_\mathbf c+\rho)$ 
  contains two equal numbers,  making   $s_\beta (\lambda_\mathbf c+\rho)$  not $\Phi_I$-regular. Consequently,  $\beta \notin\Psi^{++}_{\lambda_\mathbf c}$.
In conclusion, if we choose 
\begin{equation}
\label{equ:chooseofqtqs>1}
    q_t\ge \max\{q_s-m_{st}\mid 2\le s< t\}, \text{ for all } 3\le t\le k,
\end{equation}
 then $\beta\not\in \Psi^{++}_{\lambda_{\mathbf c}}$ for any $\beta=\epsilon_i-\epsilon_j$ with $ i \in \mathbf  p_{s}$  and $j\in \mathbf p_{t}$, $1<s<t\le k$.

Suppose $s=1$. We claim  that  either $\mathbf d(t)\subset \mathbf d(1)$ or $\mathbf d(t)\supset \mathbf d(1)$ holds if we impose suitable restrictions on $q_1, q_2, \ldots, q_k$. In either case,   the larger set contains two equal numbers. Specifically, if $(\mathbf d(t))\subset (\mathbf d(1))$, then the $1$-th part of $s_\beta(
\lambda_{\mathbf c} +\rho)$ contains two identical  entries. Conversely,  a similar result holds by considering  the $t$-th part instead of the $1$-th part. In either  case, $s_\beta(\lambda_{\mathbf c}+\rho)$ is not $\Phi_I$-regular, which ensures that   $\beta\not\in \Psi_{\lambda_\mathbf c}^{++}$. 

We verify our claim as follows. 
First, suppose  $m_{1k}\ge 0$. By \eqref{mine}, we have   $u_1-\frac{1}{2}\ge u_t-\frac{1}{2}$. We  choose 
 \begin{equation}
 \label{Equ:case2q1}
   q_1\ge \max\{q_t+m_{1t}\mid 2\le t\le k\}.   
 \end{equation} 
 As a result, we have  
 $(\mathbf d_t) \subset (\mathbf d_1)$,
 because   $  u_1+\frac{1}{2}-q_1\le   u_t+\frac{1}{2}-q_t$
 by \eqref{Equ:case2q1}.

 Next,  consider the case where   $m_{12}\le  0$. 
   We choose
\begin{equation}
\label{Equ:case3q1}
    q_t\ge \max\{q_h-m_{ht}\mid 1\le h<t\}, \text{ for } 2\le t\le k.
\end{equation} 

Finally, for the case 
$m_{12}\ge m_{13}\ge \ldots\ge m_{1h}\ge 0 > m_{1,h+1}\ge \ldots\ge m_{1k}$,   
we  choose $q_1, q_2, \ldots, q_k$ successively such that 
\begin{equation}
 \label{Equ:case3q12}
 \begin{aligned} &
    q_1\ge \max\{q_t+m_{1t}\mid 2\le t\le h \},\\  & q_t\ge \max\{ q_i-m_{it}\mid 2\le i\le t-1\} \text{ for } 2\le t\le h, \\
 & q_t\ge \max\{q_i-m_{it}\mid 1\le i\le t-1 \} \text{ for }h+1\le t\le k.
\\ 
\end{aligned}
\end{equation}  
In each of these  cases, we ensure that  either $(\mathbf d_t) \subset (\mathbf d_1)$ or $(\mathbf d_t)\supset (\mathbf d_1)$,  completing  the proof of our claim. \end{proof}

``\textbf{Proof of Theorem~B}":  If $\Psi^+_{\lambda_\mathbf c}=\emptyset$, there is nothing to  prove. Otherwise,  any  
element in  $\Psi^+_{\lambda_{\mathbf c}}$
is of form $\epsilon_i+\epsilon_j  $ or  $\epsilon_i-\epsilon_j$.  

In either case, by Lemma~\ref{posi11} or  Lemma~\ref{pos12},  we can impose additional restrictions on $\mathbf c$ along with suitable chosen  $q_1, q_2, \ldots, q_k\gg 0$ to ensure that \eqref{cjuj} holds and that  
$\Psi_{\lambda_{\mathbf c}}^{++}=\emptyset$. Consequently, by Proposition~\ref{Pro:simcri},  $M^\mathfrak p(\lambda_{\mathbf c})$ is both simple and tilting.
\qed

From this point on, we always assume that the $\mathbf c$ and $q_i$ are chosen as in Theorem~B, 
and that
$  u_{k+1},   u_{k+2}, \ldots,   u_{2k}\in \mathbb C^k$ satisfy \eqref{uij21}. 
We  define 
\begin{equation}
\label{tildeu} 
{\mathbf u}_{2k}=(u_1, u_2, \ldots, u_k, u_{k+1}, \ldots, u_{2k})\end{equation} 
where $u_{1}, u_2, \ldots, u_{2k}$ are given by 
\eqref{cjuj}-\eqref{uij21}.

\subsection{Decomposition numbers for $B_{2k,r}( \mathbf u_{2k})$ }

Recall ${\mathbf F}_{r}$ and $\bar {\mathbf F}_{r}$ \eqref{frbar}.
For any $\mu \in \Lambda^\mathfrak p$, we define 
\begin{equation} \label{hatmu}\ \ \mathbf \delta^\mu =\mu-\lambda_{\mathbf c}, \text{ and }|\delta^\mu|=\sum_{i=1}^n |\delta_i^\mu|.\end{equation}
It follows from \cite[Lemma~4.8]{RS-cyc} (see also the proof of \cite[Theorem 5.4]{RS-cyc}) that 
\begin{equation}
\label{equ:criterioninFr}
{\mathbf F}_{r}=\bigcup_{f=0}^{\lfloor r/2\rfloor}\left\{\mu\in \Lambda^\mathfrak p\mid \delta^\mu_i\in \mathbb Z , \delta^\mu_{p_{s-1}+1} \ge 
 \ldots \ge \delta^\mu_{p_s}, \forall 1\le s\le k,  |\delta^\mu|=r-2f\right\}. 
\end{equation}

Let $\tilde\  : {\mathbf F}_{r} \rightarrow \Lambda_{2k,r}$ denote the inverse of the map $\hat{} $ introduced in Section~1.   The map  $\tilde\ $ is a bijection satisfying 
\begin{equation}
\label{equ:defoftilde}
 \tilde \mu^{(j)}=\begin{cases} (\delta^\mu_{p_{j-1}+1},\ldots, \delta^\mu_{p_{j-1}+r} ) &\text{if $1\le j\le k$,}\\  (-\delta^\mu_{p_{2k-j}},\ldots, -\delta^\mu_{p_{2k-j}-r+1} ) &\text{if 
 $k+1\le j\le 2k$.}
\end{cases}   
\end{equation}
Here we identify $\tilde \mu:=(\tilde u^{(1)}, \tilde u^{(2)}, \ldots, \tilde u^{(2k)}) $ with $(f,\tilde \mu)$ since $f$ is uniquely determined by the equation $2f=r-\sum_{j=1}^{2k}|\tilde \mu^{(j)}|$.

Recall the right  cell modules $C(f,\lambda)$, and simple modules  $D(f,\lambda)$ introduced  in the introduction for $ B_{2k,r}(\mathbf u_{2k})$. The following result is a special case of \cite[Theorem D]{GRX}. Notably, $F_0$ is always saturated, as  $M^\frp(\lambda_{\mathbf c})$ is simple in our case.  
\begin{Theorem}
    \label{first1}  Suppose Condition~\ref{keyconj} holds. 
    Then,  we have  
    \begin{itemize}
    \item[(1)]  $\Hom_{\mathcal O^\frp } (M^{\frp}( \hat\lambda), M_{r,\mathbf c})\cong C( f,\lambda')$, as  right $  B_{2k, r}(\mathbf {u}_{2k})$-modules, $\forall \ (f, \lambda')\in \Lambda_{2k,r}$.  
    \item[(2)]  $D(\hat \lambda)\cong D(f, \lambda')$ for all $(f, \lambda')\in \bar{\Lambda}_{2k, r}$. 
\item [(3)]   $M_{r,\mathbf c}=\oplus_{(f, \lambda')\in \bar \Lambda_{2k,r}} T^{\frp}(\hat \lambda )^{\oplus \dim D(f,  \lambda')} $. 
\item[(4)] 
$ [C( f,\lambda'): D(\ell,\mu')]=(T^{\mathfrak p}( \hat\mu):M^{\mathfrak p}(\hat \lambda))$, 
   $\forall\  (f, \lambda)\in \Lambda_{2k, r}$, and $(\ell , \mu')\in\bar\Lambda_{2k, r}$.\end{itemize}     \end{Theorem}

\subsection{A sufficient condition for the saturated condition }
It was proved in \cite[Theorem~E]{GRX} that ${\mathbf F_{j}}$ is saturated for any $1\le j\le r$  if  
 $\lambda_{\mathbf c}$ satisfies \begin{equation}\label{keyass} \langle \lambda_{\mathbf c}+\rho, \beta^\vee\rangle\not \in \mathbb Z_{>0}, \ \ \text{ for all $\beta\in \Phi^+\setminus \Phi_I.$}\end{equation} 
In this case, $\Psi^+_{\lambda_\mathbf c}=\emptyset$.

This subsection aims to prove that   
${\mathbf F_{j}}$ is saturated for any $1\le j\le r$, if  
 \begin{equation}\label{PhiA}\Phi_A\cap \Psi_{\lambda_\mathbf c}^+=\emptyset,\end{equation} where $ \Phi_A$ is defined as  in \eqref{equ:defofphiA}. 

It is worth noting  that this condition is strictly weaker than    \eqref{keyass}. For the case  $k=1$, the equality 
\eqref{PhiA} always hold  unconditionally. 
Theorem~\ref{first1} subsequently  establishes an explicit correspondence  between parabolic Verma modules appearing as subquotients of $M_{ r, \mathbf c}$  in the parabolic  BGG category associated with the maximal parabolic subalgebra of $\mathfrak{so}_{2n}$, and  the right cell modules of $B_{2, r}(\mathbf u_{2})$.  Thus, our result strengthens the bijection established in \cite{ES2}, which lacked  an  explicit constructive characterization of right cell modules for $B_{2, r}(\mathbf u_{2})$.

\begin{Lemma}\label{lem:keyde} Suppose $\Phi_A\cap \Psi_{\lambda_\mathbf c}^+=\emptyset$. 
Let  $\mu\in {\mathbf F}_{r}$ and $\beta\in \Phi^+\setminus \Phi^+_I$
such that $\langle  \lambda_\mathbf c+\rho ,\beta^\vee\rangle\in \mathbb Z_{>0}$ and 
$\langle  \mu+\rho ,\beta^\vee\rangle\in \mathbb Z_{>0}$.
If there exists  some $w_\beta\in W_{I}$ such that $\nu=w_\beta s_\beta( \mu+\rho)-\rho\in \Lambda^\mathfrak p$, then we have  $\nu\in {\mathbf F}_{r}$.
\end{Lemma}

\begin{proof} 
Since  $\Phi_A\cap \Psi_{\lambda_\mathbf c}^+=\emptyset$, we only need  to consider $\beta=\epsilon_i+\epsilon_j$ with $i<j$, where  either  $i,j\in \mathbf p_s$,  for some $1\le s\le k$ or  $i\in \mathbf p_s$ and $j\in \mathbf p_t$, with  $1\le s<t\le k$. From this point to the end of the proof, we denote  $\lambda_\mathbf c$ by  $\lambda$ for simplicity. Recall we
identify  $\lambda+\rho$ with $(\mathbf d_1, \mathbf d_2, \ldots, \mathbf d_k)$, where $(\mathbf d_t)$ was defined as in \eqref{dt}.
Then, we have 
$$(\lambda+\rho)_l=u_s-l'+\frac{1}{2}$$ for $p_{s-1}+1\le l\le p_s$, where $l'$ is defined as  in \eqref{iprime}.

In the first case, where  $i, j\in \mathbf p_s$, we have  $ 2u_s\in \mathbb Z $ and $u_s>1$,  since 
 $\langle \lambda+\rho, \beta^\vee\rangle\in \mathbb Z_{> 0} $.
Moreover, the condition $$ (\lambda+\rho)_i+ (\lambda+\rho)_j= 2u_s+1-i'-j'>0$$
implies that $i,j<p_s-r$ since $q_s\gg0$.
This ensures that  both $\delta^\mu_i$ and $\delta^\mu _j$ are non-negative.

Next,   consider $s_\beta(\mu+\rho)$. For any admissible $h$, we have  
$$s_\beta(\mu+\rho)_h=\begin{cases} -(\mu+\rho)_i &\text{if $h=j$,}\\
- (\mu+\rho)_j &\text{if $h=i$,}\\
 (\mu+\rho)_h & \text{otherwise.}\\
 \end{cases}
 $$
Thus, the $s$-th part of $s_\beta(\mu+\rho)$ is:
$$(u_s-\frac{1}{2}+\delta^\mu_{p_{s-1}+1},  \ldots, - (\mu+\rho)_j,  \ldots,- (\mu+\rho)_i, \ldots,   u_s-q_s+\frac{1}{2}+\delta^\mu_{p_s}).$$ Since $q_s\gg0$,  for any $q_s-2r<h'\le q_s$,
we have   
$$-(\mu+\rho)_i-(\mu+\rho)_h= i'+h'-2u_s-1-\delta_i^\mu -\delta_h^\mu>0.$$ 

Notably, since  $-(\mu+\rho)_i-(\mu+\rho)_j=-\langle \mu+\rho,\beta^\vee\rangle\in \Z_{<0}$, 
there exists  some $a$ such that $(\nu+\rho)_a=-(\mu+\rho)_i $ with $j'\le a'< q_s-2r$.
Moreover, if $b\in \mathbf p_s$ and $b>a$, then 
$ (\nu+\rho)_b= (\mu+\rho)_b $, and  
$$  (\nu+\rho)_{p_{s-1}+1}\ge  (\nu+\rho)_{p_{s-1}+2}\ge \ldots \ge (\nu+\rho)_{a+1}.$$
In particular, $ ( \nu+\rho)_{ a+1}=(\mu+\rho)_{ a+1}\ge  (\lambda+\rho)_{ a+1}$. Consequently, we have   
$$
    \sum_{l\notin  \{p_{s-1}+1,\ldots, a\}} |(\nu+\rho)_l-(\lambda+\rho)_l| = \sum_{l\notin  \{p_{s-1}+1,\ldots, a\}} |(\mu+\rho)_l-(\lambda+\rho)_l|,$$
    and $$
     \sum_{l=p_{s-1}+1}^a |(\nu+\rho)_l-( \lambda+\rho)_l| = \sum_{l=p_{s-1}+1}^a  |(\mu+\rho)_l-( \lambda+\rho)_l|-2((\mu+\rho)_i+(\mu+\rho)_j ).$$
This ensures that $|\delta^\nu|=r-2f$ for some $f$ since 
$\mu\in \mathbf F_r $ and $(\mu+\rho)_i+(\mu+\rho)_j \in \Z_{>0} $. Hence, 
$\nu\in  {\mathbf F}_{r}$ by \eqref{equ:criterioninFr}.

In the second case, we assume  $\beta=\epsilon_i+\epsilon_j$    with $i\in \mathbf p_s$,  $j\in \mathbf p_t$, and   $ s<t$.
Since   $\langle  \lambda+\rho, \beta^\vee\rangle\in \mathbb Z_{> 0} $, we must have $ u_s+u_t\in \mathbb Z $ and $u_s+u_t>1$. Moreover, because $q_s,q_t\gg0$, we have $(\lambda+\rho)_i+ (\lambda+\rho)_j= u_s+u_t+1-i'-j'>0$ if   $i<p_s-r$ and $ j<p_t-r$. 
This ensures that  both  $\delta^\mu_i$ and $\delta^\mu_j$ are non-negative.  

For $ s_\beta( \mu+\rho)$, we  have 
$$s_\beta( \mu+\rho)_h=\begin{cases} - (\mu+\rho)_i &\text{if $h=j$,}\\ -(\mu+\rho)_j &\text{if $h=i$,}\\
  (\mu+\rho)_h, &\text{otherwise.}
  \end{cases}$$
So, the $s$-th and $t$-th parts of $s_\beta( \mu+\rho)$ are: 
$$\begin{aligned} & (u_s-\frac{1}{2}+\delta^\mu_{p_{s-1}+1},  \ldots, -(\mu+\rho)_j  , \ldots,   u_s-q_s+\frac{1}{2}+\delta^\mu_{p_s}), \\
& (u_t-\frac{1}{2}+\delta^\mu_{p_{t-1}+1},     \ldots,-(\mu+\rho)_i, \ldots,   u_t-q_t+\frac{1}{2}+\delta^\mu_{p_t}).\\
\end{aligned}$$ Since $q_s\gg 0$,  we have 
 $-\mu_j-\mu_h= j'+h'-u_s-u_t-1-\delta_i^\mu -\delta_h^\mu>0$
for  $q_s-2r<h'<q_s$.  Consequently,  $-(\mu+\rho)_j=  (\nu+\rho)_a$ for some $a$ such that  $i'\le a'< q_s-2r$, since $-(\mu+\rho)_i-(\mu+\rho)_j=-\langle \mu+\rho,\beta^\vee\rangle\in \Z_{<0}$. 
Similarly, we have $-(\mu+\rho)_i= (\nu+\rho)_b$ for some $b$ such that  $j'\le b'< q_t-2r$.
Moreover, for $ a, c\in \mathbf p_s$ with $c>a$, or $b,c\in \mathbf p_t$ with  $c>b$, 
we have  $  (\nu+\rho)_c= (\mu+\rho)_c $, and  
\[ \begin{aligned} & (\nu+\rho)_{p_{s-1}+1}\ge (\nu+\rho)_{p_{s-1}+2}\ge \ldots \ge (\nu+\rho)_{ a+1}, \\
& (\nu+\rho)_{p_{t-1}+1}\ge  (\nu+\rho)_{p_{t-1}+2}\ge \ldots \ge (\nu+\rho)_{ b+1}.\end{aligned} \]
In particular, $ (\nu+\rho)_{ h+1}=(\mu+\rho)_{ h+1}\ge  (\lambda+\rho)_{h+1}$  for  $h\in \{a, b\}$. 
 This ensures that: 
$$\begin{aligned}
    \sum_{l\notin  \{p_{s-1}+1,\ldots, a\}\cup \{ p_{t-1}+1,\ldots ,b\}} |\nu_l-\lambda_l| &= \sum_{l\notin  \{p_{s-1}+1,\ldots, a\}\cup \{ p_{t-1}+1,\ldots ,b\}} |\mu_l- \lambda_l|,\\
      \sum_{l \in  \{p_{s-1}+1,\ldots, a\}\cup \{ p_{t-1}+1,\ldots ,b\}} |\nu_l-\lambda_l| &= \sum_{l \in  \{p_{s-1}+1,\ldots, a\}\cup \{ p_{t-1}+1,\ldots ,b\}} |\mu_l-\lambda_l|-2y, 
\end{aligned} $$
where $y= (\mu+\rho)_i+(\mu +\rho)_j\in \Z_{>0}$.
This implies $|\delta^\nu|=r-2f'$ for some $f'$, 
and  we conclude  that  $ \nu\in  {\mathbf F}_{r}$ by \eqref{equ:criterioninFr}. 
\end{proof}

 \begin{Theorem}
\label{Thm:phidecomposition}
If $\Phi_A\cap \Psi_{\lambda_\mathbf c}^+=\emptyset$, then Condition~\ref{keyconj} is satisfied,   and hence   Theorem~\ref{first1}(1)-(4) hold. 
    \end{Theorem}
\begin{proof} We remark that all arguments in \cite[\S6]{GRX} remain valid,  except for those  
related  to \cite[Lemma 6.8]{GRX}.   However, Lemma \ref{lem:keyde} serves  as a  replacement for  \cite[Lemma~6.8]{GRX}, ensuring that  
\cite[Theorem~E]{GRX} still holds. This confirms that  ${\mathbf F}_{r}$ is saturated. As explained in \cite{GRX}, ${\mathbf F_{j}}$ is also saturated for any $1\le j<r$,  since the assumptions for $M_{r, \mathbf c}$ are automatically satisfied  for $M_{j, \mathbf c}$. Thus, Condition~\ref{keyconj} is satisfied.   The second conclusion follows immediately.   
\end{proof}

\begin{Cor}\label{first2} When $\mathfrak p$ is the maximal parabolic subalgebra of $\mathfrak{so}_{2n}$ with respect to  $I=\{\alpha_1,\ldots, \alpha_{n-1}\}$, Theorem~\ref{first1}(1)--(4)
hold. In this case, $k=1$.
\end{Cor}
\begin{proof} Since $I=\{\alpha_1, \ldots, \alpha_{n-1}\}$, we automatically have  $\Phi_A\cap \Psi_{\lambda_\mathbf c}^+=\emptyset$. Thus, Theorem~\ref{first1}(1)--(4) follow from Theorem~\ref{Thm:phidecomposition}. 
\end{proof}

\section{Decomposition numbers for cyclotomic Brauer algebra with arbitrary parameters}
Building on the results from the  previous two sections, we are now  ready to  determine the decomposition numbers of $B_{k,r}((-1)^k\mathbf u_k)$.
Throughout this section, we fix a choice of   $\mathbf u_k$ such that   $\mathbf c$ and $p_0,p_1, p_2, \ldots, p_k$ are selected  as in   Theorem \ref{Thm:simpleVerma2}. Moreover, 
${\mathbf u}_{2k}$ is chosen as in \eqref{tildeu}.
    \subsection{The isomorphism}    
Recall that $q_i=p_{i}-p_{i-1}$ for $i=1,\ldots,k$.
\begin{Lemma} \label{disj}
There exists  a   
    $\mathbf c=(c_1, c_2, \ldots, c_k)\in  \mathbb C^k$, along with the associated positive integers  $q_1,q_2,\ldots, q_k\gg0$, such that $u_{j+1}$ is disjoint from $\mathbf u_j:=(u_1, \ldots, u_j)$  for any $k\le j\le 2k-1$.    
\end{Lemma}
\begin{proof}
Write $u'_l= u_{2k-l+1}$ for $1\le l\le k$, where $u_{k+1}, \ldots, u_{2k}$ are defined in \eqref{uij21}.
Then, we have 
$$u'_l=-c_l+p_l-n+\frac{1}{2}.$$
We divide the proof into the following cases.

Suppose  $|u'_l-u_t'|\in \Z$    for some  $1\le l<t\le k$. 
If  $l>1$, then we can choose $q_t,q_l$ such that
 \[|u'_l-u_t'|=|c_t-c_l +p_l-p_t|=  |q_l-q_t+u_t-u_l |\ge r \]
 under the choices made in   \eqref{equ:chooseofqtqs>1}.
The case  $l=1$ can be handled similarly, using the choices specified in \eqref{Equ:case2q1}--\eqref{Equ:case3q12}.

Suppose $u'_l+u_t'\in \Z$
for some  $1\le l<t\le k$.  Since $q_t,q_l\gg0$, we have  
\[|u_t'+u_l'|= |-u_t-u_l+q_t+q_l | \ge r.\]

Suppose $u'_l-u_t\in\Z $ for some $1\le l, t\le k$. Since $q_l\gg0$, we have 
\[|u_l'-u_t|= |-u_t-u_l+q_l | \ge r.\]

Suppose  $u'_l+u_t\in\Z $ for some $1\le l, t\le k$. Since $q_l\gg0$,    we have 
\[|u'_l+u_t|= |u_t-u_l+q_l| \ge r .\]
In summary,  in every case,  $u_{j+1}$ is disjoint from $\mathbf u_j$. 
\end{proof}    
Henceforth, we maintain  the assumption on $\mathbf c$ and $q_1, q_2,\ldots, q_k$  in Lemma \ref{disj}. This allows us to apply  Lemma~\ref{disj} and Theorem~\ref{Thm:subisok2k1}, which  directly leads to the following result.

\begin{Theorem}\label{Equ:idemkto2k} There exists 
 an idempotent $e$ in  
 $ B_{2k,r}( \mathbf u_{2k})$ with  $  {\mathbf u}_{2k}$ given in \eqref{tildeu}, such that 
$    B_{k,r}((-1)^k\mathbf u_k) \cong e  B_{2k,r}(\mathbf u_{2k}) e$.  
\end{Theorem}

``\textbf{Proof of Theorem~C}":   As demonstrated  in  \cite[Theorem~5.4]{RS-cyc},  the required  isomorphism exists  if $M^\mathfrak p(\lambda_{\mathbf c})$ is simple, and $q_t\ge 2r$ for all $1\le t\le k$. Therefore, Theorem~C(1) follows directly from Theorem \ref{Thm:simpleVerma2}. Finally, Theorem~C(2) is a consequence of Theorem~\ref{Equ:idemkto2k} and Theorem~C(1). \qed

\subsection{Decomposition numbers of $B_{k,r}((-1)^k\mathbf u_k)$}
In this subsection, we determine the decomposition numbers of $B_{k,r}((-1)^k\mathbf u_k)$ using Theorem \ref{Equ:idemkto2k} and Theorem \ref{first1}.
Recall the bijection $\tilde \ :{\mathbf F}_{r}\rightarrow \Lambda_{2k,r}$ as defined in \eqref{equ:defoftilde}.
\begin{Lemma}\label{frk} For any positive integer $k$,  ${\mathbf F}_{r,k}=\{\mu \in {\mathbf F}_{r}\mid \delta^\mu_i\ge 0 \text{ for all } 1\le i \le n  \}$
where ${\mathbf F}_{r,k}$ is defined as  in \eqref{barfk}. \end{Lemma}
 \begin{proof} 
 Recall from the proof of  \cite[Theorem 5.14]{RS-cyc} that a bijection exists  between the set of Verma paths in $M_{r, \mathbf c}$  of length $r$ terminating  at  $M^\frp(\lambda)$, and   the set $
     \mathscr T_{2k,r}^{up}(\tilde\lambda)$  of all up-down tableaux of type $\tilde  \lambda$, where $\tilde \lambda$ denotes 
      the $2k$-partition corresponding to $\lambda$ under the map defined  as   in \eqref{equ:defoftilde}.  
     
     For each  $\t=(\t_0,\t_1, \ldots, \t_r)\in \mathscr T_{2k,r}^{up}(\tilde \lambda) $, we associate  a sequence of contents
     $\text{cont}(\t)=(b_1,\ldots, b_r)$, defined by 
     \begin{equation}
     \label{equ:defofcont}
      b_i=\begin{cases}   
         u_t+ h-l & \text{if $\t_i\setminus \t_{i-1}=p$,}\\
         -(u_t+h-l) & \text{if $\t_{i-1}\setminus \t_{i}=p $,}\\
         \end{cases}   
     \end{equation}
    where  $p= (l,h, t)$.
    
Notably,  $M_{r, \mathbf c}$ admits  a finite parabolic Verma flag, and the operator $ x_j$ acts  on $M_{r,\mathbf c}$ via $\Omega+n-\frac{1}{2} $ on $M_{j-1,\mathbf c}\otimes V $, where $ \Omega=\frac{1}{2}\sum_{i,j\in [n]}F_{i,j}\otimes F_{j,i}$,   the Casimir element in $\U(\mathfrak g)\otimes \U(\mathfrak g) $ (see \cite[(3.12)]{RS-cyc}). Consequently,    $M_{r, \mathbf c}$ decomposes into a direct sum of   common generalized eigenspaces for  the operator $x_1, x_2, \ldots, x_r$ as follows;
 \begin{equation}
     \label{equ:decom}
     M_{r,\mathbf c}=\oplus_{\mathbf i\in \C^{r}} (M_{r,\mathbf c})_{\mathbf i}, 
 \end{equation}
 where  $(M_{r,\mathbf c})_{\mathbf i}$
 denotes  the generalized eigenspace for the eigenvalue tuple $\mathbf i$. 
 
We claim that the parabolic Verma module $ M^\frp(\lambda)$,  associated with the Verma path specified  by $\t\in \mathscr T_{2k,r}^{up}(\tilde\lambda) $,  arises  as a subquotient of the direct summand
$(M_{r,\mathbf c})_{\mathbf b}$
 in \eqref{equ:decom}, where $\mathbf b=(b_1, b_2, \ldots, b_r)$ is the content of $\t$, defined as  in \eqref{equ:defofcont}.

Indeed, as explained earlier,  by \cite[Theorem 3.8]{RS-cyc}  $ x_j$ acts on $M_{r,\mathbf c}$ via   $\Omega+n-\frac{1}{2} $ on $M_{j-1,\mathbf c}\otimes V $.
For  $0\le j\le r-1 $, let $\nu\in {\mathbf F_{j}}$ such that $\t_j=\tilde \nu$. Then 
$\t_{j+1}= \tilde \mu  $, where  $\mu=\nu+\sigma \epsilon_i$ for some $i$, $\sigma\in\{\pm 1\}$.
By \cite[(4.15)]{RS-cyc}, the operator  $x_{j+1}$ acts   
on $M^\frp(\mu)$ via the scalar:
\begin{equation}
\label{equ:aj+1}\begin{aligned} 
a_{j+1}&=\frac{1}{2}( ( \nu+\sigma \epsilon_i, \nu+\sigma \epsilon_i+2\rho)- (\nu,\nu+2\rho)-(\epsilon_1, \epsilon_1+2\rho)+n-\frac{1}{2} )\\
&=\sigma \nu_i+ (\sigma \epsilon_i-\epsilon_1, \rho)+n-\frac{1}{2}.\end{aligned} 
 \end{equation}

Suppose that $\sigma=1$. Then either $\t_{j+1}\setminus \t_j =p $ or $\t_{j}\setminus \t_{j+1} =p $ for some $p$. 
According to the definition of the bijection $\tilde \ $, we have 
$$\begin{cases}   1\le t\le k,  & \text{ if $\t_{j+1}\setminus \t_j =p $ and $p=(h,l, t)$,}\\
k+1\le t\le 2k,  & \text{ if $\t_{j}\setminus \t_{j+1} =p $ and $p=(h,l, t)$.}\\
\end{cases}$$
 
In the first case,  $i=p_{t-1}+h$ and $\nu_i=(\lambda_{\mathbf c})_i+ l-1$. Thus, by \eqref{equ:aj+1} we have 
  $$\begin{aligned}
  a_{j+1}&= (\lambda_\mathbf c)_i+l-1 -p_{t-1}-h+1+n-\frac{1}{2}\\
&= c_t+l-p_{t-1}-h+n-\frac{1}{2}\\
&\overset{\eqref{cjuj}} =u_t+l-h =b_{j+1},     
  \end{aligned}$$
which proves  the claim in this case.
Analogous reasoning applies to
  $ \t_{j+1}= \t_j\setminus p$ or $\sigma=-1$. 
This completes the proof of the claim.

Suppose that $\lambda\in {\mathbf F}_{r}\setminus {\mathbf F}_{r,k}$. Then, there exists some $k+1\le j\le 2k$ such that  $u_j$ appears as an eigenvalue for some $x_i$ acting on $M^\frp(\lambda)$ associated with   any Verma path, as  stated   in the claim. 
From the claim and the definition of the idempotent $e\in B_{2k,r}(\mathbf u_{2k})^{op}\cong \End_{\mathcal O}(M_{r,\mathbf c})$, it follows  that $M^\frp(\lambda)$ (associated to any Verma path) is annihilated  by $e$ since   $u_j$ is disjoint from $u_1, u_2, \ldots, u_{j-1}$ for $k+1\le j\le 2k$ by Lemma \ref{disj}. 
Consequently,  $S(\lambda)e=\Hom_{\mathcal O^\frp}(M^\frp(\lambda), M_{r, \mathbf c})e =0$ for any  $\lambda \in {\mathbf F}_{r}\setminus {\mathbf F}_{r,k}$.

Now suppose that $\lambda\in {\mathbf F}_{r,k}$. In this case, there is a Verma path $\t$ for $M^\frp(\lambda)$ such that each component   $\t_i$ of $\t$ is a multipartition with the last $k$ parts being $\emptyset$. By the claim, it follows  that the parabolic Verma module $M^\frp(\lambda)$ associated with this Verma path belongs to $(M_{r,\mathbf c})_\mathbf i$ with each $i_j$ is disjoint from $u_{k+1}, \ldots, u_{2k}$.  Thus, $S(\lambda)e=\Hom_{\mathcal O^\frp}(M^\frp(\lambda), M_{r, \mathbf c})e \neq0 $. 

This completes the proof.
 \end{proof}

Using standard arguments on idempotent functors, we deduce that $$\{D(\mu)e\mid \mu\in \bar {\mathbf F}_{r,k}\}$$ forms a complete set of pairwise non-equivalent irreducible modules for  $B_{k,r}((-1)^k\mathbf u_k)$, where $\bar {\mathbf F}_{r,k}$ is defined  in \eqref{barfk}. Applying \cite[Corollary~5.10]{RS-cyc} and utilizing the  idempotent functor determined by $e$ in \eqref{Equ:idemkto2k}, we obtain   the following result.  
\begin{Theorem}
\label{Mainthm}
For any  $(\lambda, \mu)\in  {\mathbf F}_{r,k} \times  \bar {\mathbf F}_{r,k}$,  
      $[S(\lambda)e: D(\mu)e]=(T^{\mathfrak p}( \mu):M^{\mathfrak p}( \lambda))$.
\end{Theorem}

By applying Theorems~\ref{Mainthm} and~\ref{first1}, we derive  the explicit decomposition numbers of $B_{k,r}((-1)^k\mathbf u_k)$, provided  Condition~\ref{keyconj} holds when $r$ is odd or $r$ is even, and $\omega_i\neq 0$ for some  $0\le i\le k-1$. Having established these results, we now proceed to prove  Theorem~D  and Theorem~E. 
 
 ``\textbf{Proof of Theorem \ref{first123}}": 
Under the assumption that Condition~\ref{keyconj} holds,   Theorem \ref{first1}(1)(2) yields  
 the isomorphisms 
 \begin{equation}\label{equ:isomordxcell}
 S(\hat\lambda)\cong C(f, \lambda'), \ \ \text{ and $D(\hat \mu)\cong D(\ell,\mu')$}
 \end{equation}
 for all  $(f, \lambda')\in \Lambda_{2k,r}$ and $(\ell, \mu')\in \bar \Lambda_{2k, r}$. Here  the bijection  $\hat \ : \Lambda_{2k,r}\rightarrow {\mathbf F}_{r}$, defined as  in \cite[(4.4)]{GRX}, serves as  
  the inverse of the map 
 $\tilde \ $ given in \eqref{equ:defoftilde}.

Recall the Jucys-Murphy basis of $B_{2k, r}(\mathbf u_{2k})$, defined in \cite[Proposition 5.8]{RSi} as 
\[\{m_{\s,\t} \mid \s,\t \in \mathscr {T}^{up}_{2k,r}(\lambda), (f,\lambda)\in \Lambda_{2k,r}\}.  \]
  We identify   $\mathscr {T}^{up}_{k,r}(\lambda) $ as a  subset of
$\mathscr {T}^{up}_{2k,r}(\lambda) $ by embedding   $\lambda\in \Lambda_{k,r}$ into $\Lambda_{2k,r}$ via $(\lambda, \emptyset,\ldots, \emptyset)$.

In the proof of Theorem A, we established  that  $e=\gamma_1'\ldots \gamma_r'$, where each $\gamma_i'$ is the idempotent projecting  onto the generalized eigenspace corresponding  to $x_i$ (for  $1\le i\le r$)  with    eigenvalues disjoint from $u_{j}$ for $k+1\le j\le 2k$.  
By   \cite[Theorem 5.12]{RSi}, this implies  
$$ em_{\s,\t}e=\begin{cases}
     m_{\s,\t} & \text{if  $\s,\t\in \mathscr {T}^{up}_{k,r}(\lambda)$}\\
    0  & \text{otherwise.}
\end{cases} $$
Applying  $e$ to the Jucys-Murphy basis of $B_{2k,r}(\mathbf u) $ produces  a cellular basis of $ B_{k,r}((-1)^k\mathbf u_k)$.
An analogous  result holds for the
cell modules. Combining this  with  Lemma \ref{frk},  we conclude  the proof of Theorem \ref{first123}(1).

 The argument further  shows that  $C(f, \lambda')$
    admits  a non-trivial bilinear form 
      if $C(f, \lambda')e$ does. 
     Additionally,  by \cite[Theorem~3.12]{RSi}, the cardinality of the set 
      $$\{D(f, \lambda')\neq 0\mid \lambda\in {\mathbf F}_{r,k} \}$$ equals  the number of  $\mathbf u$-restricted partitions when $r$ is odd or $r$ is even and $\omega_i\neq 0$ for at least one  $0\le i\le r-1$. This count coincides with  the number of non-equivalent irreducible modules of $ B_{k,r}((-1)^k\mathbf u_k)$ (see  \cite[Theorem 3.12]{RSi}). Consequently, $D(f,\lambda')e\neq 0$
if and only if $(f,\lambda') \in\bar \Lambda_{2k,r} $, which establishes the claims  (2) and (3).   Finally, claim 
(4) follows directly  from  \eqref{equ:isomordxcell} and Theorem \ref{Mainthm}.\qed


``\textbf{Proof of Theorem \ref{mainthm:E}":}   
This result  is guaranteed by Theorem \ref{Thm:phidecomposition} whenever $\Phi_A\cap \Psi^+_{\lambda_\mathbf c}=\emptyset$. \qed

\end{document}